\documentclass[12pt]{article}
\usepackage{verbatim}
\usepackage{latexsym}
\usepackage{amsmath,amsthm,amsfonts}
\usepackage{eucal}
\usepackage{cases}

\usepackage{mathtext}
\usepackage[cp1251]{inputenc}
\usepackage[T2A]{fontenc}
\usepackage[dvips]{graphicx}
\usepackage{amsmath}
\usepackage{amssymb}
\usepackage{amsxtra}
\usepackage{latexsym}
\usepackage{ifthen}
\usepackage{color}
\def\R {\mathbb{R}}
\def\C {\mathcal{C}}

\newtheorem{theorem}{Theorem}
\newtheorem{definition}{Definition}[section]
\newtheorem{lemma}{Lemma}[section]
\newtheorem{proposition}{Proposition}[section]
\newtheorem{remark}{Remark}[section]
\newtheorem{example}{Example}[section]

\numberwithin{equation}{section}

\begin{document}

\title{Qualitative analysis of solutions \\ for a degenerate PDE model\\ of epidemic dynamics}

\author{
Marina Chugunova$^1$, Roman Taranets$^2$\\ and Nataliya Vasylyeva$^{3,4,5}$\\
{\small $^1\,$Claremont Graduate University, }\\
{\small 150 E. 10th Str., Claremont, California 91711, USA}\\
{\small \textup{\texttt{marina.chugunova@cgu.edu}}}
\\
{\small $^2\,$Institute of Applied Mathematics and Mechanics of  NASU,}\\
{\small G.Batyuka Str. 19, 84100, Sloviansk, Ukraine}\\
{\small \textup{\texttt{taranets\_r@yahoo.com}}}\\
{\small $^3\,$Institute of Applied Mathematics and Mechanics of NASU,}\\
{\small G.Batyuka Str. 19, 84100 Sloviansk, Ukraine}\\
{\small $^4\,$Institute of Hydromechanics of NASU,}\\
{\small Zhelyabova Str. 8/4, 03057 Kyiv, Ukraine}\\
{\small $^5\,$Dipartamento di Matematica, Politecnico di Milano}\\
{\small Via Bonardi 9, 20136 Milano, Italy}\\
{\small \textup{\texttt{nataliy\_v@yahoo.com}}}}

\maketitle

\begin{abstract}
Compartmental models are widely used in mathematical epidemiology to describe dynamics of infection disease. A new SIS-PDE model, recently derived by Chalub and Souza, is based on a diffusion-drift approximation of   probability density  in a well-known discrete -- time Markov chain SIS-DTMC model. This new SIS-PDE model is conservative due to degeneracy of the diffusion term at the origin. We analyze a class of degenerate PDE models  and obtain sufficient conditions  for existence of classical solutions with certain regularity properties. Also, we  show that under some additional assumptions about coefficients and initial data classical solutions vanish at the origin at any finite time. 
Vanishing at the origin of solutions is consistent with the conservation property of the model. 
 The main results of this article are: sufficient conditions for existence of weak solutions, analysis of  their asymptotic behavior at the origin, and the proof of existence of weak solutions convergent to Dirac delta function.   
 Moreover, we  study long-time behavior of solutions and confirm our analysis by numerical computations.
\end{abstract}

\section{ Introduction }

 Regardless of the modern improvement in development of highly efficient antibiotics and vaccines, infectious diseases still contribute significantly to deaths worldwide. While the earlier recognized diseases like cholera or the plague still
sometimes create problems in underdeveloped countries erupting occasionally in
epidemics, in the developed countries new diseases are emerging like AIDS (1981)
or hepatitis C and E (1989-1990).  New threads constantly appear like the recent bird
flu (SARS) epidemic in Asia, the very dangerous Ebola virus in Africa and worldwide spread of COVID-19. Overall, infectious diseases continue to be one of the most important health problems. Modeling of epidemiological phenomena has a very long history with the first model for smallpox formulated by Daniel Bernoulli in 1760. From the early twentieth century, in response to epidemics of various infectious diseases, a large number of models has been constructed and analyzed, see for example \cite{al1994, al2000, ba1975, di2013, ew2004, fe2005} (and references therein).

Compartmental models are a very general modeling technique. They are often applied to the mathematical modeling of infectious diseases. The population is assigned to compartments like: susceptible, infectious, and recovered in widely used SIR model
or to: susceptible, infective, and susceptible like in SIS epidemiological scheme. SIS scheme provides the simplest description of the
dynamics of a disease that is contact-transmitted, and that does not lead to immunity like for example COVID-19. Discrete-time Markov chain type  SIR and SIS models are considered to be a classical approaches in modern mathematical modeling in epidemiology. The most recent development in mathematical epidemiology is based on introduction of a continuous modeling based on partial differential equations SIR-PDE model like in \cite{CH2011,CH2014,CH2009a,CH2009b,CH2013}.

In our paper, for $T>0$ and $\Omega=(0,1),$ $\Omega_{T}=\Omega\times(0,T)$, we 
 consider the following equation (see \cite{CH2014}) in the unknown function $p:=p(x,t)$: $\bar{\Omega}_{T}\to\R$:
\begin{equation}\label{c-1}
\frac{\partial p}{\partial t} = \frac{1}{2N}\frac{\partial^{2}}{\partial x^{2}} ( f(x) p) - \frac{\partial }{\partial x}(g(x) p) \quad\text{ in }\quad \Omega_{T},
\end{equation}
coupled with the boundary condition
\begin{equation}\label{c-2}
\frac{1}{2N}\bigg[ (1-R_0) p(1,t) + \frac{\partial}{\partial x}p(1,t) \bigg] + p(1,t) = 0,  \quad t\in[0,T],
\end{equation}
and initial data
\begin{equation}\label{c-3}
  p(x,0) = p_0(x) \quad \text{in} \quad \bar{\Omega}.
\end{equation}
Addressing to the SIR-PDE model, we conclude that  $x \in \bar{\Omega}  $ is the fraction of infected, $N$ is a population of individuals,
$p $ is the probability to find a fraction $x$  at time $t$ in a population of size $N$,
$R_0 \geqslant 0$ is the basic reproductive factor,
$$
f(x) := x(R_0(1-x) +1), \ \  g(x) := x(R_0(1-x) - 1).
$$
This model is degenerate at $x = 0$ that is a boundary point of the domain. 


It is worth noting that processes defined by similar models were studied by Feller in the early 1950s and used to great effect by Kimura, et al. in the 1960s and 70s to give
quantitative answers to a wide range of questions in population genetics. Although, rigorous analysis of analytic properties of these equations is only now
 getting into the focus of applied mathematicians.  The study of initial or/and initial-boundary value problems for degenerate equations including Kimura-type operators has a long history.  We do not provide here a complete survey on the published results concerned to these degenerate equations, but rather present some of them. Indeed, the investigation of elliptic and parabolic problems to degenerate equations  containing the operators like
\[
\mathcal{L}:=a(x)\sum_{ij=1}^{n}a_{ij}(x)\frac{\partial^{2}}{\partial x_{i}\partial x_{j}}+\sum_{i=1}^{n}b_{i}(x)\frac{\partial}{\partial x_{i}}
\] 
with $a(x)\approx |x|^{\alpha},$ $\alpha>0,$ and $a_{ij}$ satisfying ellipticity conditions, are extensively studied by many authors with various analytical approaches (see e.g. \cite{BK,Fi,FP1,FP2,FP3,MP,NP,PPT1,PPT2,PT1,Pu,SV,Ve}) including stochastic calculus 
\cite{ABBP,BP,Ch,CE,EP2011,EM2,EP2,Ku,Po2}.

Under suitable assumptions on the asymptotic behavior of the operator's coefficients at the boundary of the domain, the uniqueness of bounded solutions was shown in \cite{Fi,PT1} without prescribing any boundary conditions  at the origin. 
 In \cite{PPT2}, taking advantage of appropriate super- and subsolutions, similar uniqueness results have been proved in the case of unbounded solution. In \cite{NP,Pu,Pu2}, under some conditions on the coefficients, 
uniqueness results for Cauchy problem to parabolic equation like $\frac{\partial u}{\partial t}-\mathcal{L}u=f$ in the bounded domain in suitable weighted $L^{2}$ spaces are established.

 Epstein et al. \cite{EP2011,EM2,EM3,EP1,EP2} (see also references therein) study the generalized Kimura operators (generalization $\mathcal{L}$ with $\alpha=1$) in the manifold setting. In particular, they discussed maximum principle and the Harnack inequality. The integral maximum principles in the whole $\R^{n}$ for solutions of degenerate parabolic equations are also obtained in \cite{AB}. As for probabilistic approaches, there are abundant theories on existence and uniqueness of solutions to stochastic differential equations with degenerate diffusion coefficients (see \cite{CE,EP2011,EM2,Et,Ku}), besides well-posedness of the related problems in the case of $\alpha=1$ are discussed in \cite{ABBP,BP,Ch}. It is worth noting that, degenerate diffusion is examined in the context of measure-valued process (see \cite{CC,DM,Pe}) as well as via the semigroup techniques \cite{BFR,GL,Po1}. As for well-posedness of parabolic degenerate problems, we quote \cite{BD,BK,De,FP1,FP3,Po2,PPT1,SV,Ve}, where existence of weak and classical solutions is established for different value of $\alpha>0$. In particular, if $\alpha=1$ the solvability of Cauchy problem in semi-space $\R^{n}_{+,T}$ in the weighted H\"{o}lder classes with boundary and without boundary condition on the part of boundary are discussed \cite{BK, EP2011, Po1}, while the case of $\alpha>1$ is analyzed in \cite{CHEN2008,De,Ve}. 
Previous researchers like \cite{CHEN2008,EP2011} restricted their attention to the
solutions with the best possible regularity properties, which leads to considerable simplifications. For real applications it is important to consider
solutions with more complicated behavior that is the goal of our paper.

Outline of the paper is as follows: section 2 contains the main notation and functional setting; in section 3, we show existence of stationary solutions, confirm numerically their meta-stability and analyze convergence; in section 4, we analyze particular classical and weak solutions; in section 5, we study a local asymptotic of solutions at the origin and discuss sufficient conditions which provide the fulfillment of the conservation law.


\section{Functional setting and notations}
\label{s2}
\noindent
Throughout this work, the symbol $C$ will denote a generic positive constant, depending only on the structural quantities of the problem. We will carry out our analysis in the framework of the weighted H\"{o}lder and Sobolev spaces. To this end, throughout the paper, let
\[
\alpha\in(0,1)\quad\text{and}\quad s\in\R
\]
be arbitrarily fixed.

\noindent For any non-negative integer $k$, and any Banach space $(\mathbf{X},\|\cdot\|_{\mathbf{X}}),$ and any $p\geq 1$, we consider spaces
$$
\C([0,T],\mathbf{X}),\quad L^{p}([0,T],\mathbf{X}),\quad \C^{k+\alpha,\frac{k+\alpha}{2}}(\bar{\Omega}_{T}).
$$
The last class so-called the parabolic H\"{o}lder space has been used by several authors, and its definition and properties can be found, for instance, in \cite[(1.10)-(1.12)]{LSU}. Along the paper, we will also encounter the usual spaces $\C^{k+\alpha}(\bar{\Omega}),$ $W^{1,p}(\Omega)$,  $L^{p}(\Omega)$ and $L_{\omega}^2(\Omega)$. It is worth noting that the last class is a weighted space $L^{2}$ with a  weight $\omega$  and induced norm
$$
\|v\|_{L_{\omega}^2(\Omega)}= \bigg(\int \limits_{\Omega}\omega(x)v^2(x)\,dx  \bigg)^{1/2} . 
$$
Moreover, we will use the notations $H^{1}(\Omega)$ and $H_{0}^{1}(\Omega)$ for $W^{1,2}(\Omega)$ and $W_{0}^{1,2}(\Omega)$, respectively.

Let $d(x)$ be the distance from any point $x\in\bar{\Omega}\subset\R$ to the origin and set
\[
\underline{d}=\min\{d(x),d(\bar{x})\},\quad \bar{d}=\max\{d(x),d(\bar{x})\},\quad x,\bar{x}\in\bar{\Omega},
\]
\[
d=\begin{cases}
\underline{d}\quad\text{for}\quad s\geq 0,\\
\bar{d} \quad\text{for}\quad s<0.
\end{cases}
\]
Denoting for any $\gamma\in(0,1)$,  and $x\neq\bar{x},$ $t_1\neq t_2$
\begin{align*}
\langle v\rangle_{x,s,\Omega}^{(\gamma)}&=\underset{\bar{\Omega}}{\sup}\, d^{-s}\underline{d}^{\gamma}\frac{|v(x)-v(\bar{x})|}{|x-\bar{x}|^{\gamma}},\\
\langle v\rangle_{x,s,\Omega_T}^{(\gamma)}&=\underset{\bar{\Omega}_{T}}{\sup}\, d^{-s}\underline{d}^{\gamma}\frac{|v(x,t)-v(\bar{x},t)|}{|x-\bar{x}|^{\gamma}},\\
\langle v\rangle_{t,s,\Omega_T}^{(\gamma)}&=\underset{\bar{\Omega}_T}{\sup}\, d^{-s}\frac{|v(x,t_1)-v(x,t_2)|}{|t_1-t_2|^{\gamma}},
\end{align*}
we give the following definition:
\begin{definition}\label{d2.1}
Functions $v=v(x)$ and $w=w(x,t)$ belong to the spaces $\C^{k+\alpha}_{\alpha}(\bar{\Omega})$ and $\C_{\alpha}^{k+\alpha,\frac{k+\alpha}{2}}(\bar{\Omega}_{T})$, for $k=0,1,2,$ if the norms here below are finite:
\begin{align*}
\|v\|_{\C_{\alpha}^{k+\alpha}(\bar{\Omega})}&=\sum_{j=0}^{k}\Big[\Big\|d^{\lfloor j/2\rfloor }(x)\frac{\partial^{j} v}{\partial x^{j}}\Big\|_{\C(\bar{\Omega})}
+\Big\langle d^{\lfloor j/2\rfloor }(x)\frac{\partial^{j} v}{\partial x^{j}}\Big\rangle_{x,\alpha/2,\Omega}^{(\alpha)}
\Big],\\
\|w\|_{\C_{\alpha}^{k+\alpha,\frac{k+\alpha}{2}}(\bar{\Omega}_T)}&=\|w\|_{\C([0,T],\C^{k+\alpha}_{\alpha}(\bar{\Omega}))}+
\sum_{j=0}^{k}
\Big\langle \frac{\partial^{j} w}{\partial x^{j}}\Big\rangle_{t,0,\Omega_T}^{(\frac{k+\alpha-j}{2})},\quad k=0,1,
\\
\|w\|_{\C_{\alpha}^{2+\alpha,\frac{2+\alpha}{2}}(\bar{\Omega}_T)}&=\|w\|_{\C([0,T],\C^{2+\alpha}_{\alpha}(\bar{\Omega}))}+
\|\partial w/\partial t\|_{\C^{\alpha,\frac{\alpha}{2}}_{\alpha}(\bar{\Omega}_{T})}\\
&+
\Big\langle \frac{\partial w}{\partial x}\Big\rangle_{t,0,\Omega_T}^{(\frac{\alpha}{2})}
+
\sum_{j=1}^{2}
\Big\langle d(x)\frac{\partial^{j} w}{\partial x^{j}}\Big\rangle_{t,0,\Omega_T}^{(\frac{2+\alpha-j}{2})},
\end{align*}
where $\lfloor j \rfloor $ denotes the floor function of $j$ (i.e. the greatest integer less than or equal to $j$).
\end{definition}

It is apparent that in any domain $\Omega_{\varepsilon}=(\varepsilon,l)\subset\Omega$, the spaces $\C_{\alpha}^{k+\alpha}(\bar{\Omega}_{\varepsilon})$ and $\C^{k+\alpha,\frac{k+\alpha}{2}}_{\alpha}(\bar{\Omega}_{\varepsilon,T})$ coincide with the usual H\"{o}lder $\C^{k+\alpha}(\bar{\Omega}_{\varepsilon})$  and the parabolic H\"{o}lder spaces
 $\C^{k+\alpha,\frac{k+\alpha}{2}}(\bar{\Omega}_{\varepsilon,T})$.

Moreover, in our analysis,   the following weighted H\"{o}lder spaces are needed.
\begin{definition}\label{d2.2}
Functions $V:=V(x)$ and $W:=W(x,t)$ belong to the spaces $E_{s}^{k+\alpha}(\bar{\Omega})$ and
$E_{s}^{k+\alpha,\frac{k+\alpha}{2}}(\bar{\Omega}_{T})$, for $k=0,1,2,$ if the norms here below are finite:
\begin{align*}
\|V\|_{E_{s}^{k+\alpha}(\bar{\Omega})}&=\sum_{j=0}^{k}\Big[\Big\|d^{-s+j}(x)\frac{\partial^{j}V}{\partial x^{j}}\Big\|_{\C(\bar{\Omega})}
+\Big\langle\frac{\partial^{j}V}{\partial x^{j}}\Big\rangle_{x,s-j,\Omega}^{(\alpha)}
\Big],\\
\|W\|_{E_{s}^{k+\alpha,\frac{k+\alpha}{2}}(\bar{\Omega}_{T})}&=\|W\|_{\C([0,T],E_{s}^{k+\alpha}(\bar{\Omega}))}
+
\sum_{j=0}^{k}
\Big\langle\frac{\partial^{j}W}{\partial x^{j}}\Big\rangle_{t,s-j,\Omega_{T}}^{(\frac{k+\alpha-j}{2})},
\end{align*}
for $k=0,1$, while for $k=2$
\begin{align*}
\|W\|_{E_{s}^{2+\alpha,\frac{2+\alpha}{2}}(\bar{\Omega}_{T})}&=\|W\|_{\C([0,T],E_{s}^{2+\alpha}(\bar{\Omega}))}
+\|\partial W/\partial t\|_{E_{s}^{\alpha,\frac{\alpha}{2}}(\bar{\Omega}_{T})}\\&
+
\sum_{j=1}^{2}
\Big\langle \frac{\partial^{j}W}{\partial x^{j}}\Big\rangle_{t,s-j,\Omega_{T}}^{(\frac{2+\alpha-j}{2})}.
\end{align*}
\end{definition}
 \begin{definition}\label{d2.3}
For $k=0,1,2,$ we define $\C_{\alpha,0}^{k+\alpha,\frac{k+\alpha}{2}}(\bar{\Omega}_{T})$ to be the space consisting of those functions $w\in\C_{\alpha}^{k+\alpha,\frac{k+\alpha}{2}}(\bar{\Omega}_{T})$, satisfying the zero initial conditions
\[
w|_{t=0}=0\quad\text{and}\quad \frac{\partial w}{\partial t}|_{t=0}=0,\quad m=0,...,\lfloor k/2\rfloor.
\]
\end{definition}
In a similar manner  we introduce the spaces: $\C_{\alpha}^{k+\alpha,\frac{k+\alpha}{2}}(\partial\Omega_{T})$ and
$\C_{\alpha,0}^{k+\alpha,\frac{k+\alpha}{2}}(\partial\Omega_{T})$ for $k=0,1,2$.


\section{Weak solutions: convergence to steady state and asymptotics as $T\to+\infty$}

In this section, as it is mentioned in Introduction, we discuss the long time behavior of a weak solution to problem \eqref{c-1}-\eqref{c-3}. To this end, we first construct the explicit stationary solution $P_{s}:=P_{s}(x):\bar{\Omega}\to\R$ related to \eqref{c-1}-\eqref{c-3}, and when we examine the given data which provides the convergence of the weak solution as $T\to+\infty$. In particular, we consider the case of convergence $p(x,t)$ to $P_{s}(x)$.

\subsection{Existence of a steady state}
\begin{figure}
     \includegraphics[width = 0.49\textwidth]{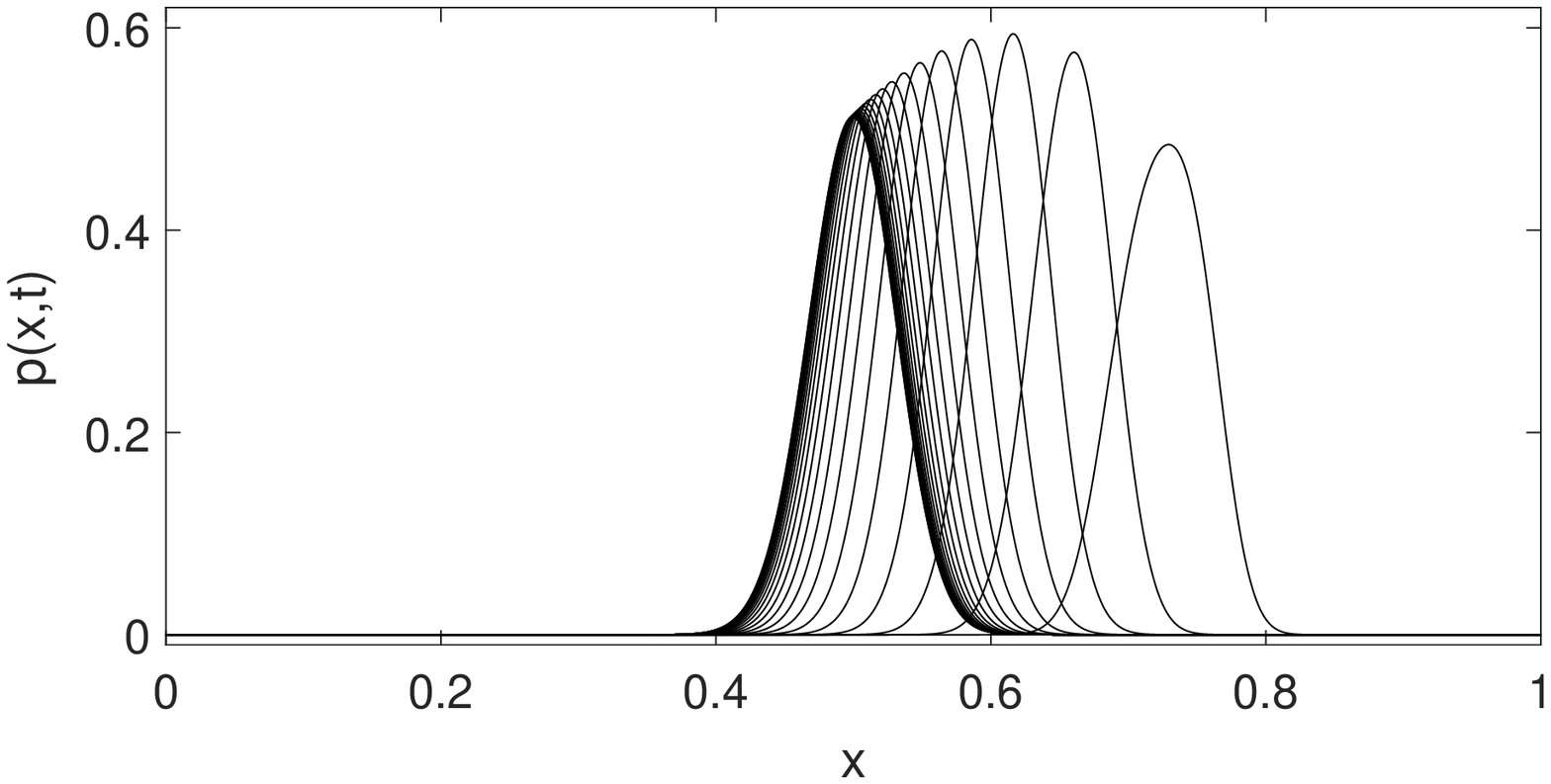}
    \includegraphics[width = 0.49\textwidth]{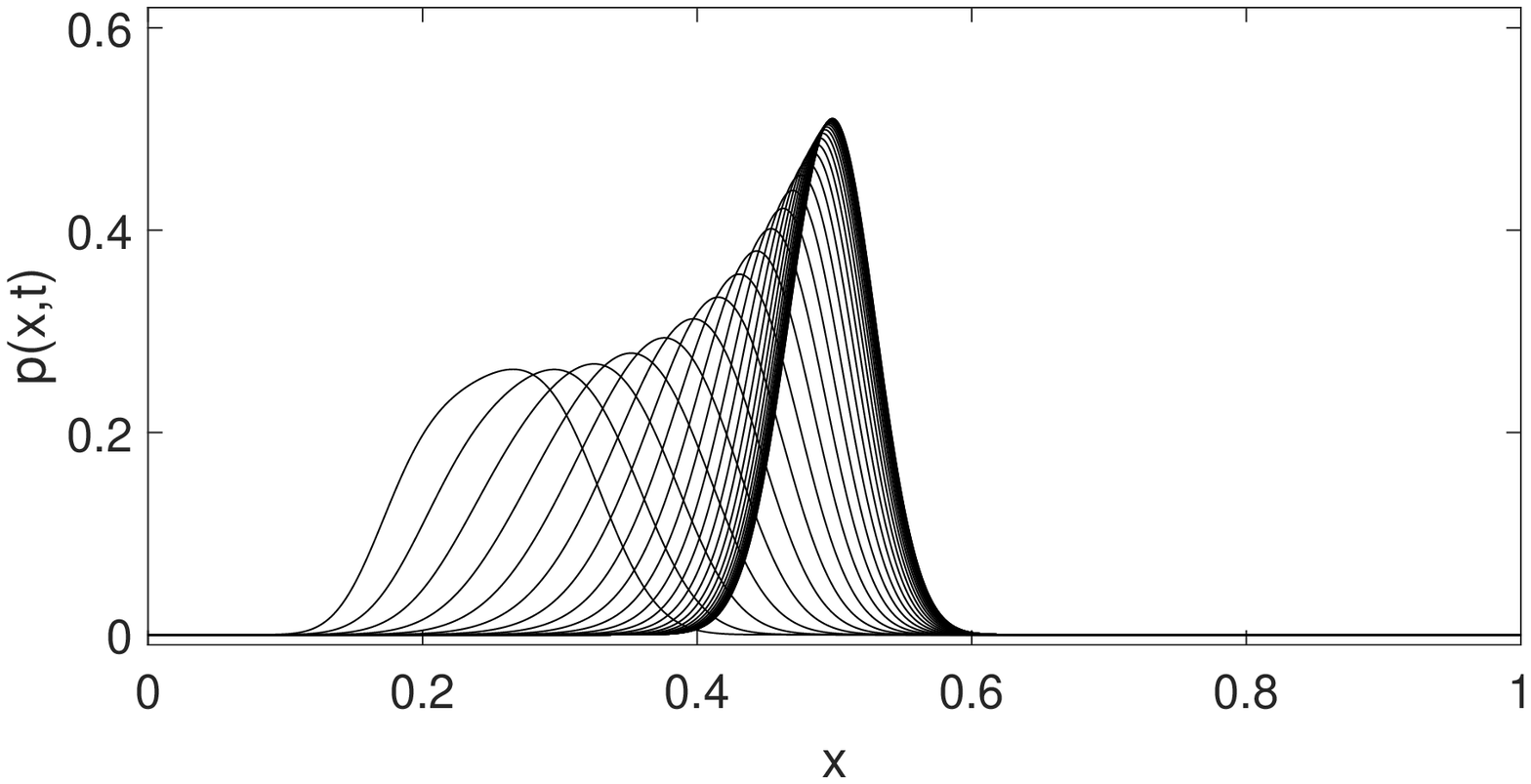}
    \caption{These two pictures illustrate the dominant behavior of convection for short times. (Left) convection moves solution toward the steady state from the right side to the left one for  $R_0 = 2$ and $N = 200$, (Right) convection moves solution toward the steady state from the left side to the right one for the same parameter values.}
    \label{steadyr1a}
\end{figure}

\begin{figure}
     \includegraphics[width = 0.49\textwidth]{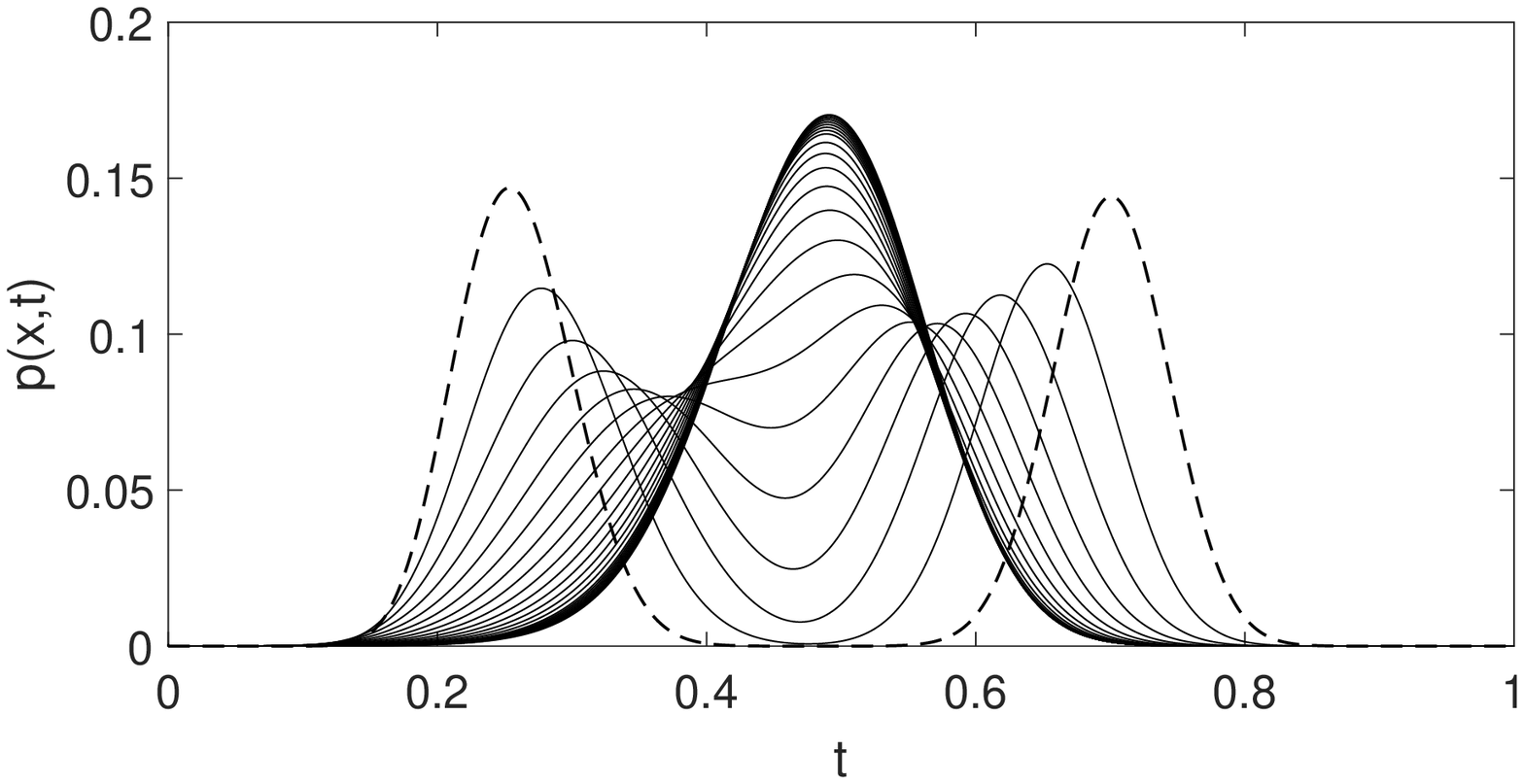}
    \includegraphics[width = 0.49\textwidth]{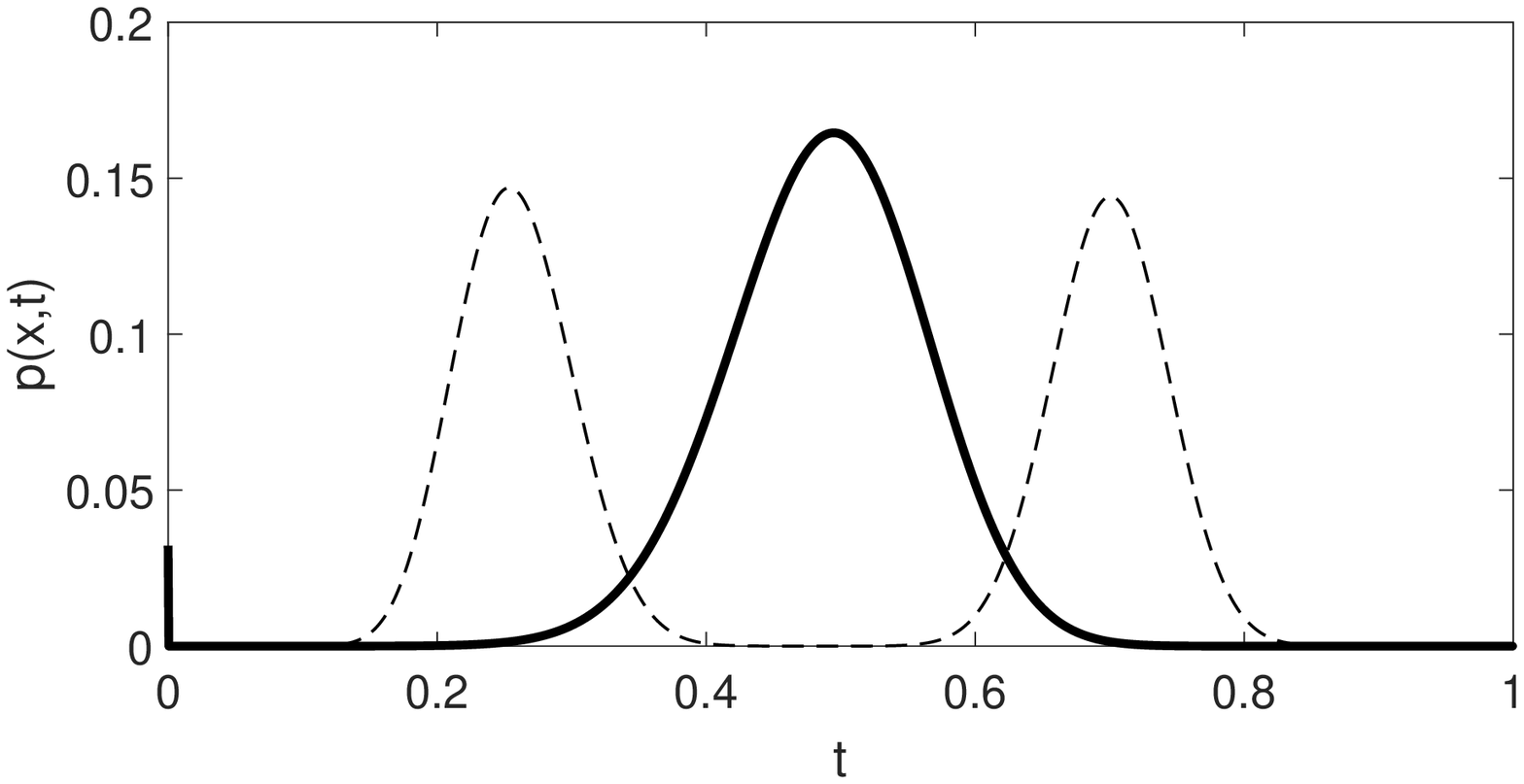}
    \caption{These two pictures illustrate short-time dynamics for $R_0 = 0$ and $N = 100$ (on the left) and long time with blow up at the origin (on the right). The initial data are plotted with a dashed line.}
    \label{steadyr3a}
\end{figure}

\begin{figure}
     \includegraphics[width = 0.49\textwidth]{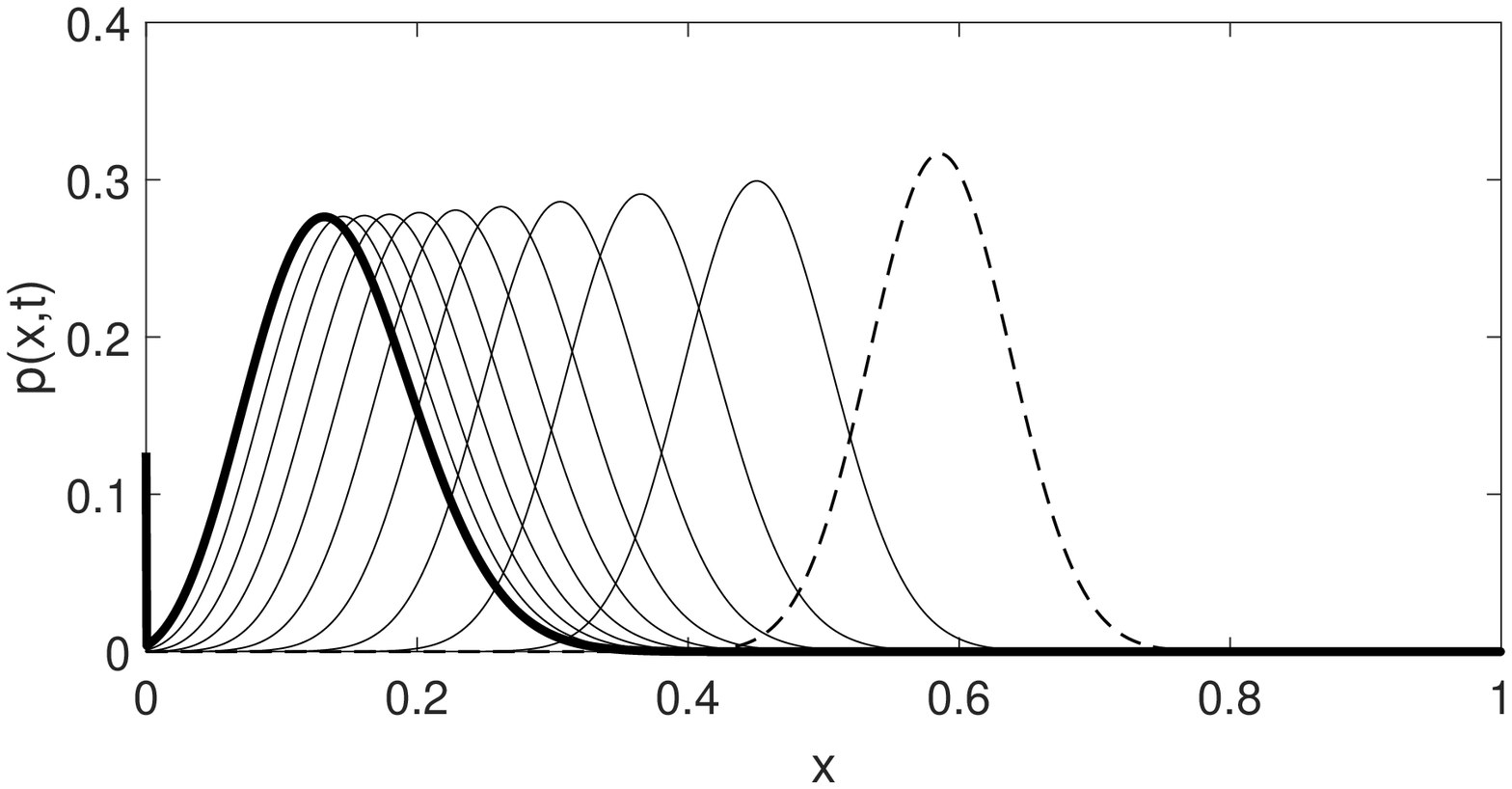}
    \includegraphics[width = 0.49\textwidth]{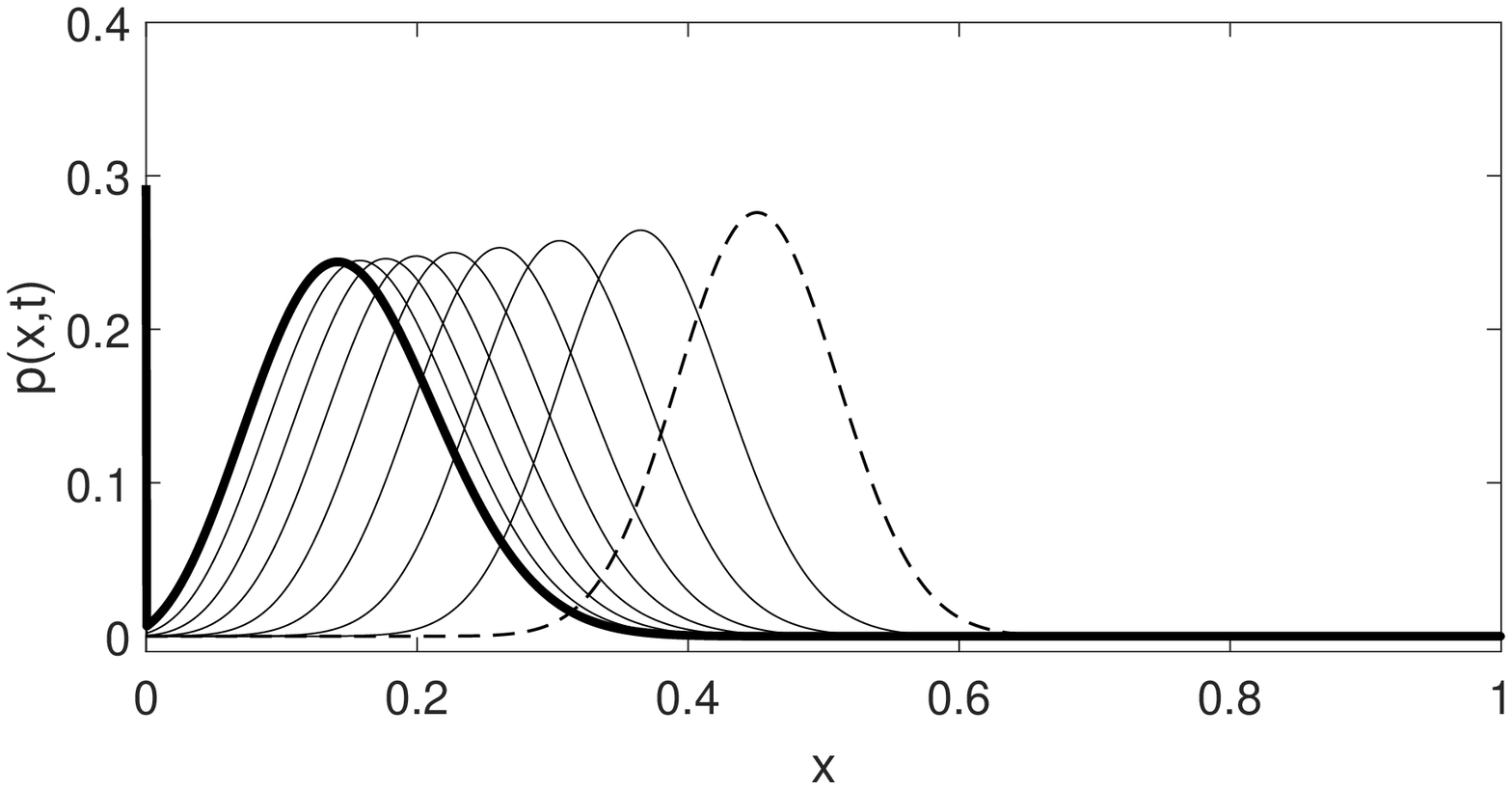}
    \caption{These two pictures illustrate the dominant behavior of convection for short times. (Left) convection moves solution toward the origin $R_0 = 0.5$ and $N = 100$, where solution blows up. (Right) convection again moves solution toward the origin  $R_0 = 1$ and $N = 100$, where it blows up.}
    \label{steadyr2a}
\end{figure}

\begin{figure}
     \includegraphics[width = 0.49\textwidth]{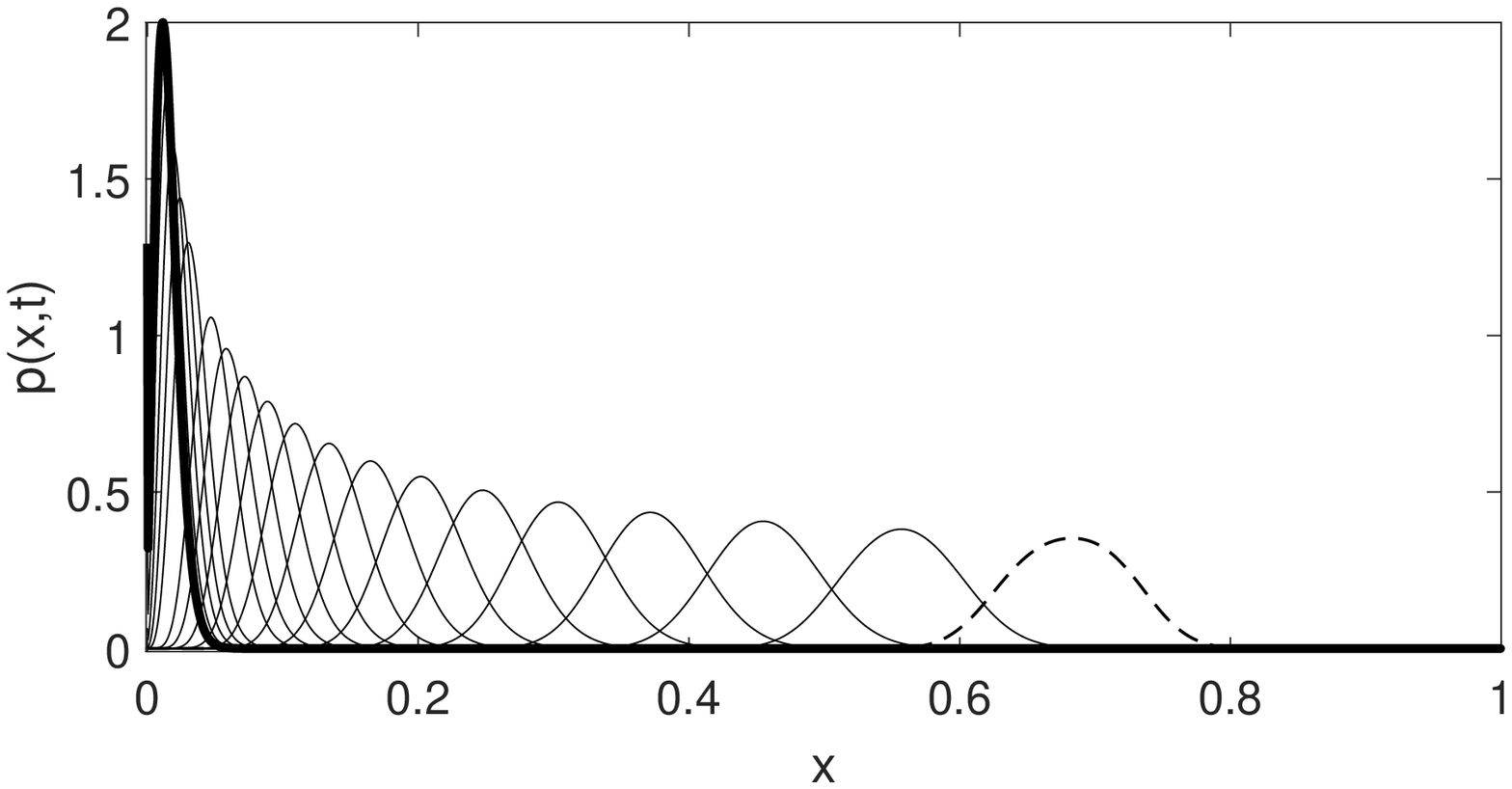}
    \includegraphics[width = 0.49\textwidth]{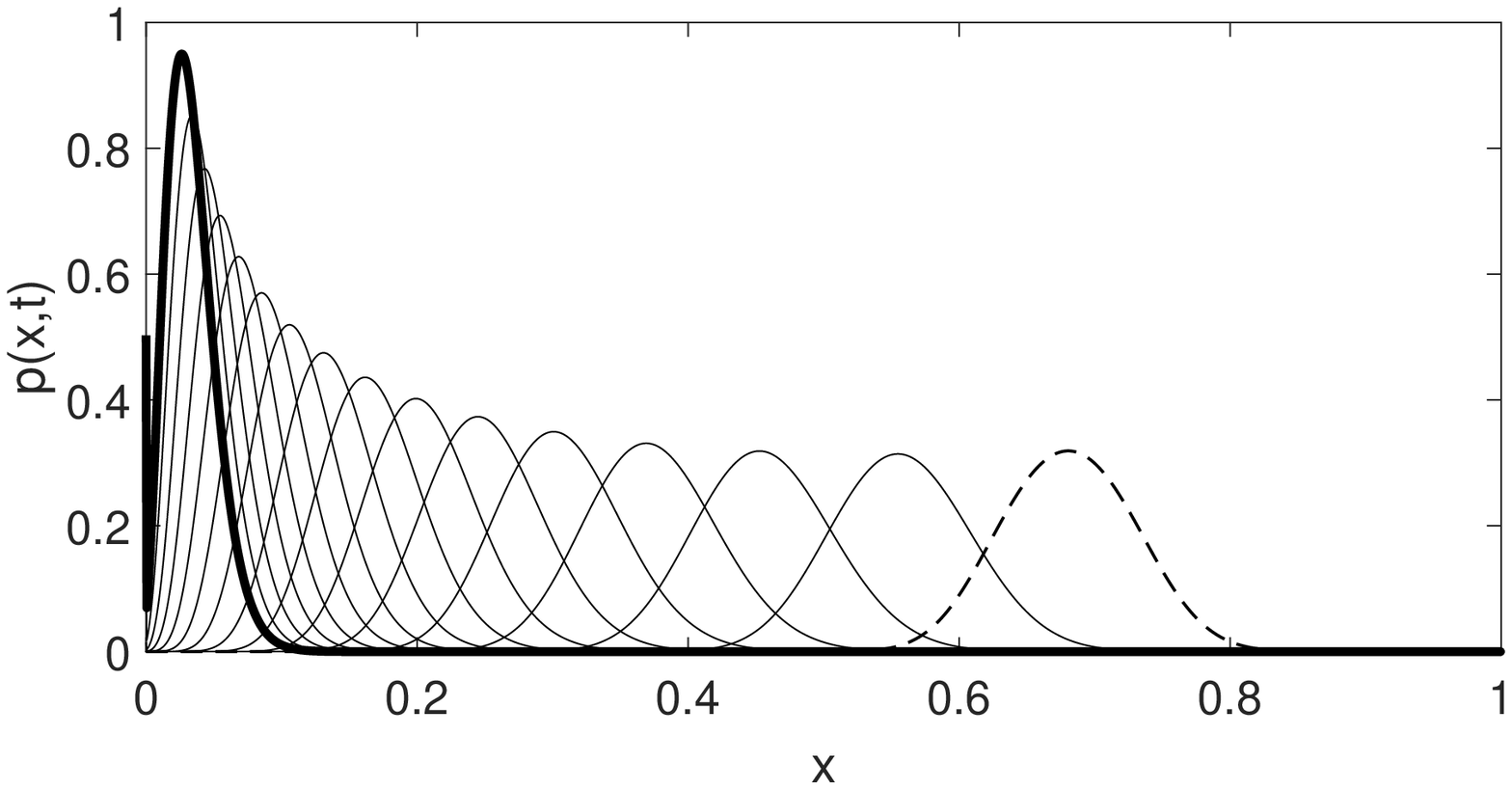}
    \caption{These two pictures illustrate the dominant behavior of convection for short times. (Left) convection moves solution toward the origin $R_0 = 0$ and $N = 100$, where solution blows up. (Right) convection again moves solution toward the origin  $R_0 = 0$ and $N = 200$, where it blows up.}
    \label{steadyr4a}
\end{figure}

First, we will start with getting an analytical formula for a stationary solution for (\ref{c-1}):
\begin{equation}\label{ss-1}
\frac{1}{2N}\frac{\partial}{\partial x}\bigl( \frac{\partial}{\partial x}(f(x)P_{s}) - 2N g(x)P_{s}  \bigr)= 0\quad\text{in}\quad \Omega,
\end{equation}
coupled with the boundary condition:
\begin{equation}\label{ss-2}
  \frac{\partial}{\partial x}P_s(1) = -(2N -R_0 +1)P_{s}(1)   .  
\end{equation}
Integrating (\ref{ss-1}) in $x$ and taking into account (\ref{ss-2}), we have
$$
\frac{\partial}{\partial x}(f(x) P_{s}) = 2N  g(x) P_{s}  .
$$
It is apparent that this equation has a general solution
\begin{equation}\label{rt-00}
f(x) P_s(x)  = \mathcal{C} \, F (x)  ,
\end{equation}
where we put
$$
F (x) := e^{ 2N \int \limits_0^x { \tfrac{g(s)}{f (s)}ds} },
\ \ \mathcal{C} := \mathop {\lim} \limits_{x \to 0} f(x) p(x).
$$
As a result, we obtain the explicit form of the classical stationary solution to \eqref{c-1}-\eqref{c-3}
\begin{equation}\label{rt-03}
P_s(x) =  \tfrac{\mathcal{C}}{\omega (x) } \text{ where }  \omega (x) := \tfrac{f(x)}{F(x)}.
\end{equation}
Hence, the changing sign convection term for $R_0 = 2$ equals to zero at $x = 0.5$ a wave-like solution moves toward this point forming a meta-stable steady state shape. The illustration that the solution short-time behavior is driven by the convection, as can be seen in Figures \ref{steadyr1a}--\ref{steadyr3a}. It takes a long time for a meta-stable steady state to move mass toward the origin, this long-time dynamics is due to a slow diffusion effect and eventually the solution eventually blows-up at the origin. That, for two different sets  of parameter values, as can be seen in Figures \ref{steadyr2a}--\ref{steadyr4a}. All numerical simulations show high accuracy of mass conservation property even for long time computations that
suggests an existence a delta function type weak solution which acts as a global attractor in this dynamical system.



\subsection{Long-time behavior of a weak solution}

Assuming that $\omega(x)$ is defined with \eqref{rt-03} and 
\[
N \geqslant 1,\quad R_0 \geqslant 0\quad \text{and}\quad 0 \leqslant p_0(x) \in L_{\omega}^2(\Omega),
\]
(see Section \ref{s2} as to the definition of $L_{\omega}^2(\Omega)$), we define a weak solution of \eqref{c-1}-\eqref{c-3} in the following sense.
\begin{definition}\label{d.bis1}
A non-negative function $p(x,t) \in L^{\infty}( 0,T ; L_{\omega}^2(\Omega))$ is a weak solution of  problem (\ref{c-1})--(\ref{c-3}) if
$$
p_t  \in L^2(0,T; H^{-1}(\Omega)), \ \ ( \omega(x) p)_x \in L^2(Q_T),
$$
and satisfies (\ref{c-1}) in the sense
$$
\int \limits_0^T { \Big\langle  \frac{\partial p}{\partial t}, \psi \Big\rangle_{H^1,H^{-1}}   \,  dt} +  \\
\iint \limits_{Q_T} {\bigl( \frac{1}{2N} \frac{\partial f p}{\partial x}   - g p  \bigr) \frac{\partial\psi}{\partial x } \,dx dt}  = 0
$$
for all $\psi \in L^2 (0,T;  H_0^1(\Omega) )$. Here $\langle  u,v \rangle_{H^1,H^{-1}}$ is a dual pair of  elements   $u \in H^1$  and $ v \in  H^{-1}$.
\end{definition}
Now we are ready to state our first main result related to 
 asymptotic behavior of a weak solution to \eqref{c-1}-\eqref{c-3}.
\begin{theorem}\label{Th-d}
(i) Let $0 \leqslant p_0(x) \in L_{\omega}^2(\Omega)$ and $ \mathop {\lim} \limits_{x \to 0} \omega(x) p(x,t) = 0 $
then  a weak solution $p(x,t)$ satisfies the relation
$$
\omega^{\frac{1}{2}}(x) p(x,t) \to 0 \text{ strongly in } L^2(\Omega)  \text{ as } t \to +\infty.
$$
Moreover, if $(\omega(x) p_0(x))_x \in L^2(\Omega) $ then there is convergence
\begin{equation}\label{con-00}
\omega(x) p(x,t) \to 0 \text{ strongly  in } W^{1,2}(\Omega) \text{ as } t \to +\infty.
\end{equation}
(ii)  Let $ \omega^{\frac{1}{2}}(x) p_0(x) \in L_{\omega}^2(\Omega)$ and $ \mathop {\lim} \limits_{x \to 0} \omega(x) p(x,t) = \mathcal{C}  > 0 $
then for a weak solution $p(x,t)$ we have
\begin{equation}\label{con-01}
\omega (x) p(x,t) \to \mathcal{C}  \text{ strongly in } L^2(\Omega)  \text{ as } t \to +\infty.
\end{equation}
Moreover, if $\omega(x)(\omega(x) p_0(x))_x \in L^2(\Omega) $ then
$$
\omega(x)\frac{\partial}{\partial x}(\omega(x) p(x,t)) \to 0 \text{ strongly  in } L^2(\Omega) \text{ as } t \to +\infty.
$$
\end{theorem}

\begin{figure}
     \includegraphics[width = 0.49\textwidth]{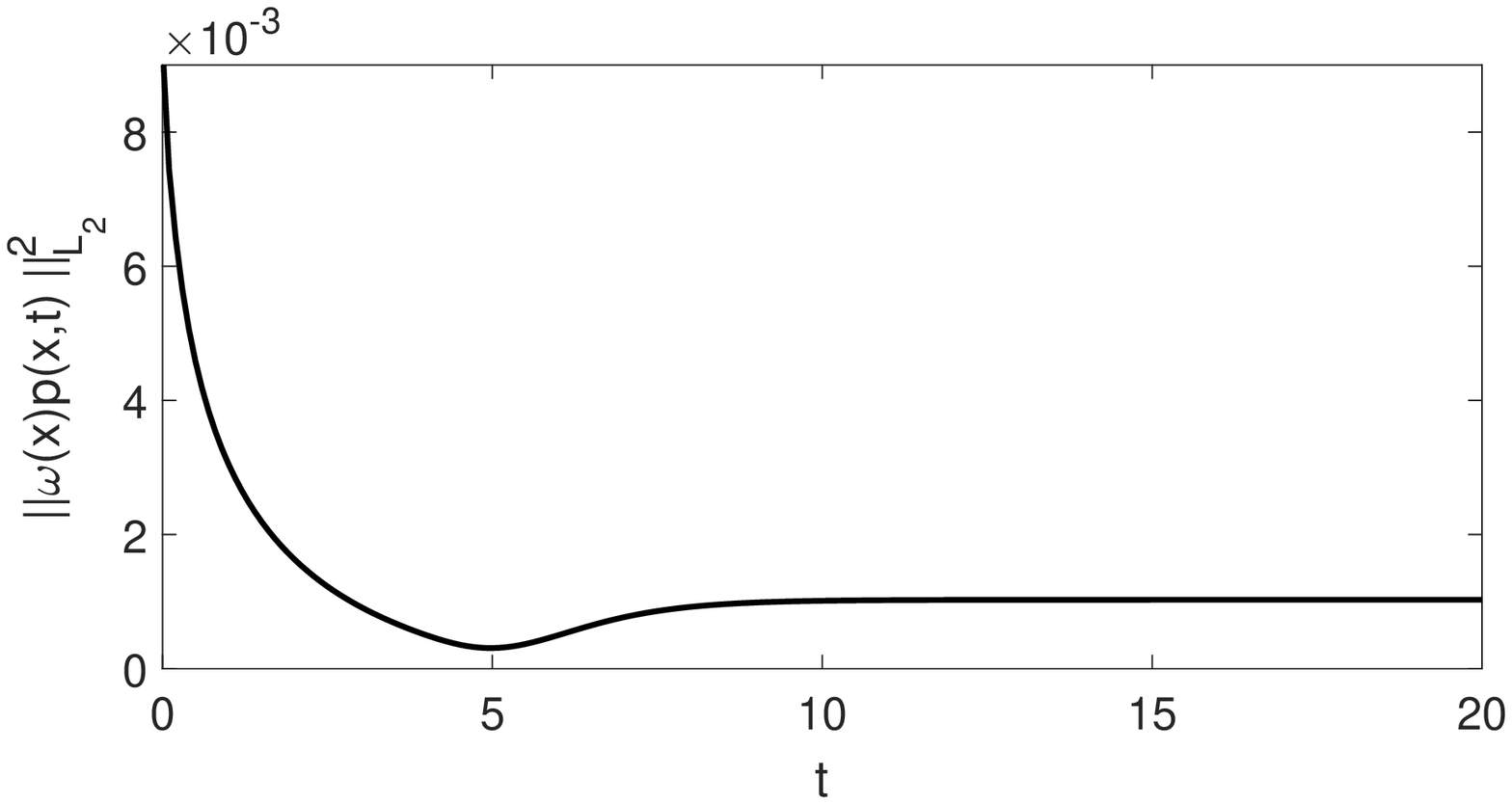}
    \includegraphics[width = 0.49\textwidth]{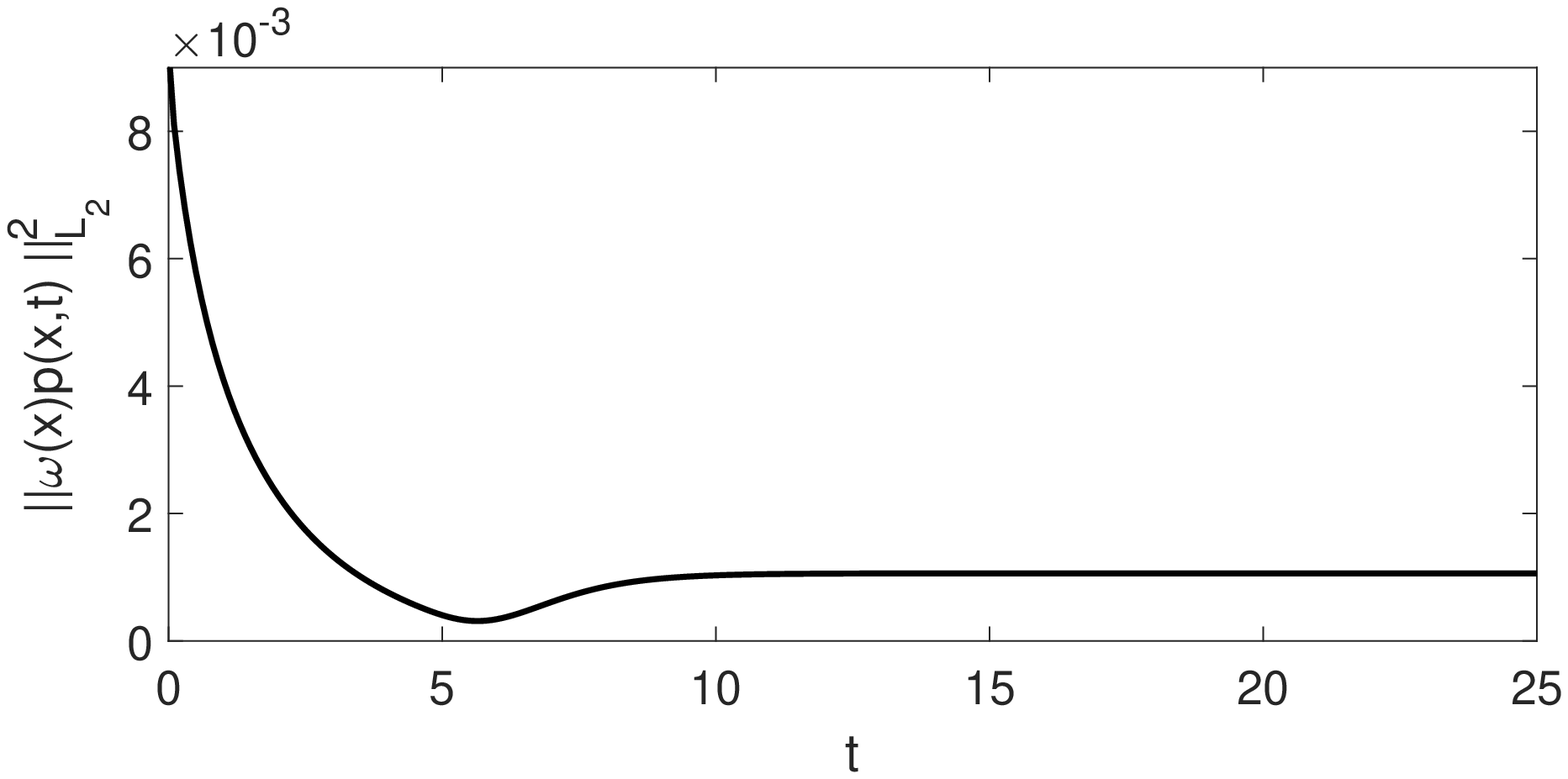}
    \caption{These two pictures illustrate convergence of weighted $L^2$ norm of $p(x,t)$ to a constant for $R_0 = 0$ and $N = 100$ (on the left) and $R_0 = 0$ and $N = 200$ (on the right).}
    \label{steadyr5a}
\end{figure}

\begin{figure}
     \includegraphics[width = 0.49\textwidth]{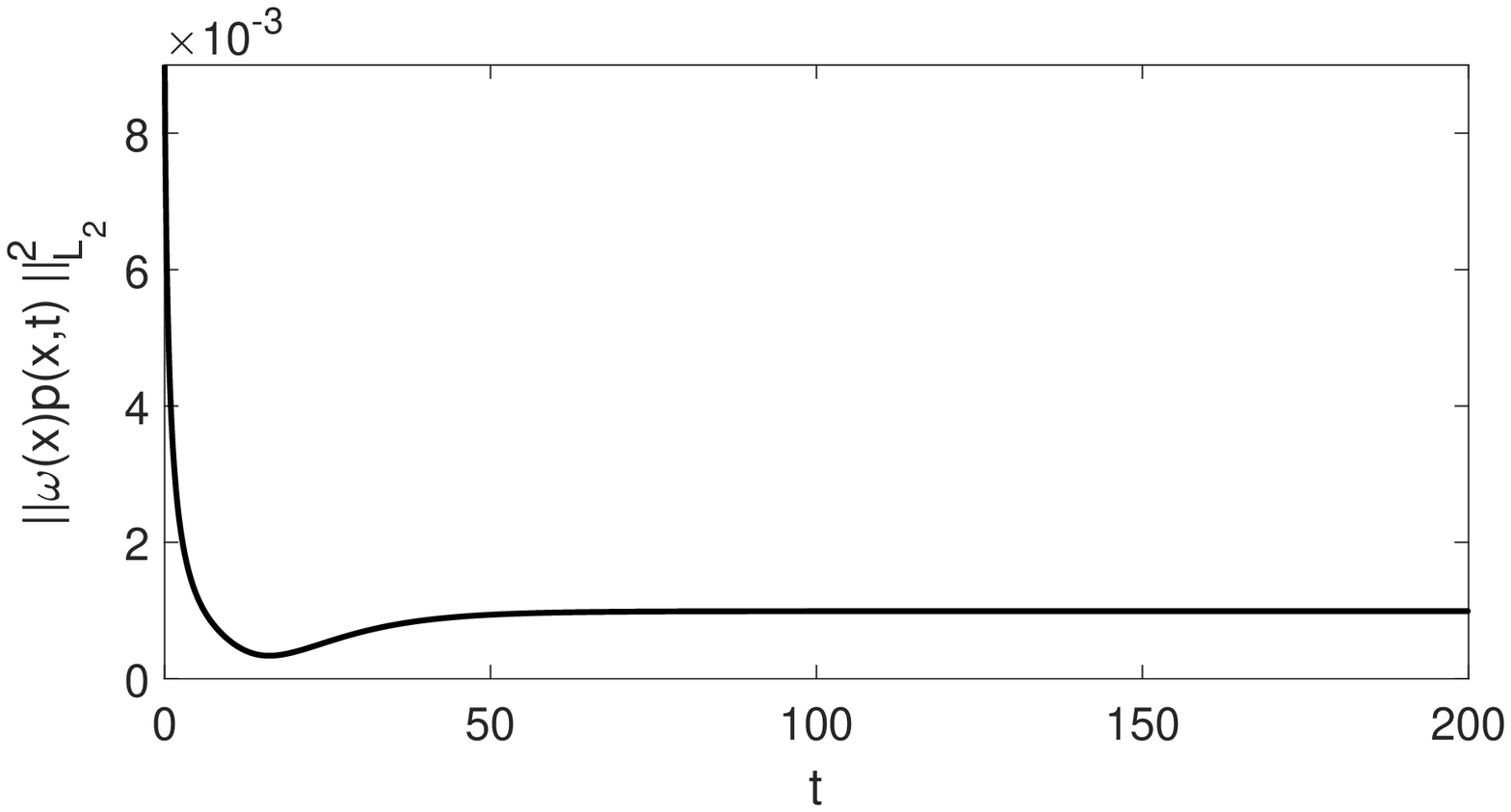}
    \includegraphics[width = 0.49\textwidth]{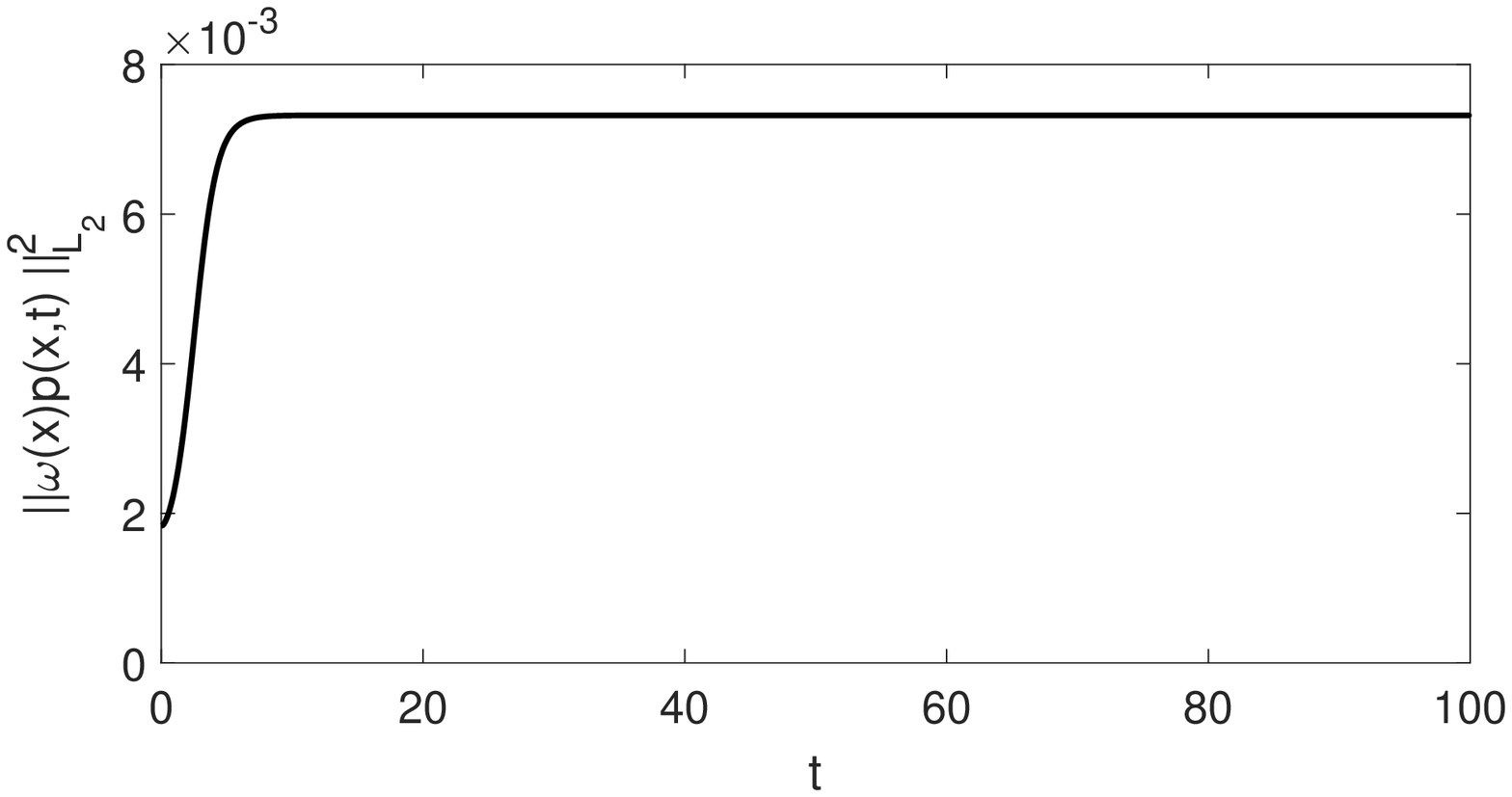}
    \caption{These two pictures illustrate convergence of weighted $L^2$ norm of $p(x,t)$ for  $R_0 = 0.5$ and $N = 100$ (on the left) and  $R_0 = 2$ and $N = 200$ (on the right).}
    \label{steadyr6a}
\end{figure}

Numerical simulations in Figures \ref{steadyr5a}--\ref{steadyr6a} illustrate the convergence result (\ref{con-01}).

\begin{remark}
Note that (\ref{con-00}) implies
$$
\int \limits_{\Omega}{p(x,t)\,dx} \to 0 \text{ as } t \to +\infty,
$$
whence we deduce that $p_{\infty} := p(x,+\infty) = 0$ for a.\,e. $x \in \bar{\Omega}$.
On the other hand, (\ref{con-01}) does the convergence to the steady state, i.\,e. $p_{\infty} = P_s$.
\end{remark}

\begin{proof}[Proof of Theorem~\ref{Th-d}]
Introducing a new function $z: = \omega(x) p(x,t),$  we rewrite  problem (\ref{c-1})--(\ref{c-3})  in more suitable form:
\begin{equation}\label{ch-1}
\begin{cases}
\omega^{-1}(x) \frac{\partial z}{\partial t} = \frac{1}{2N}  \frac{\partial}{\partial x}\big(F(x) \frac{\partial z}{\partial x}\big),\qquad\quad
(x,t)\in\Omega_{T},\\
\frac{\partial z}{\partial x}|_{x=1} = 0 \quad \text{ and }\quad z|_{x=0} = 0,\qquad  t\in[0,T],\\
z(x,0) =  \omega(x) p_0(x):=z_0(x),\qquad\quad x\in\bar{\Omega}.
\end{cases}
\end{equation}
Note that if $z|_{x=0}= \mathcal{C}  > 0$ then for the new function
$\tilde{z} = z - \mathcal{C}$  we again get problem  similar to (\ref{ch-1}).

We begin to  verify the point \textbf{(i)} of the theorem. To this end, multiplying the equation in (\ref{ch-1}) by $  z(x,t) $ and integrating  over
$\Omega$, we obtain
\begin{equation}\label{ch-3}
\frac{1}{2} \frac{d}{dt}\int \limits_{\Omega}{
\omega^{-1}(x) z^2  \,dx}  + \tfrac{1}{ 2N} \int \limits_{\Omega}{
F(x) \bigg(\frac{\partial z}{\partial z}\bigg)^2 \,dx}  =  \frac{1}{2N} F(x) z \frac{\partial z}{\partial z}  \Bigl |_0^1 = 0.
\end{equation}
Next, we take advantage of the Hardy inequality:
\begin{equation*}\label{hardy-2}
\int \limits_{\Omega}{
\omega^{-1}(x) v^2  \,dx} \leqslant C_H(R_0) \int \limits_{\Omega}{
F(x) \bigg(\frac{\partial v}{\partial x}\bigg)^2 \,dx} 
\end{equation*}
with $v(0) =0$.
Here the constant $ C_H(R_0)$ satisfies the inequalities
$$A(R_0) \leqslant C_H(R_0) \leqslant 4 A(R_0)$$ 
with
$$
A(R_0) =    \mathop {\sup} \limits_{r \in (0,1)}  \Bigl(\int \limits_{0}^r {
\tfrac{dx}{F(x)} }  \Bigr) \Bigl(\int \limits_{r}^1 {
\tfrac{dx}{\omega(x)}  }  \Bigr).
$$
Thus, collecting 
 (\ref{ch-3}) with the Hardy inequality, we end up with relations
\begin{equation}\label{ch-3-1}
 \int \limits_{\Omega}{
\omega^{-1}(x) z^2(x,t)  \,dx} \leqslant e^{- \frac{t}{  N C_H(R_0)} } \int \limits_{\Omega}{
\omega^{-1}(x) z_0^2(x)  \,dx} \to 0 \text{ as } t \to +\infty.
\end{equation}
Multiplying the equation in  (\ref{ch-1}) by $- \omega(x) \frac{\partial}{\partial x}\big(F(x) \frac{\partial z}{\partial x}\big)$ and integrating  over
$\Omega$, we obtain the equality
\begin{align*}
&\frac{1}{2} \frac{d}{dt}\int \limits_{\Omega}{
F(x) \bigg(\frac{\partial z}{\partial x}\bigg)^2  \,dx}  + \frac{1}{2N} \int \limits_{\Omega}{
\omega(x) \bigg(\frac{\partial}{\partial x}\big(F(x)\frac{\partial z}{\partial x}\big)\bigg)^2 \,dx}\\
& = F(x) \frac{\partial z}{\partial t}\frac{\partial z}{\partial x}     \Bigl |_0^1 ,
\end{align*}
which in turn provides
  \begin{equation}\label{ch-3-2}
\frac{1}{2} \frac{d}{dt}\int \limits_{\Omega}{
F(x) \bigg(\frac{\partial z}{\partial x}\bigg)^2  \,dx}  + \frac{1}{2N} \int \limits_{\Omega}{
\omega(x) \bigg(\frac{\partial}{\partial x}\big(F(x)\frac{\partial z}{\partial x}\big)\bigg)^2 \,dx}
=0. 
\end{equation}
In order to handle the second term in the left-hand side of this equality, we will exploit
 the inequality
$$
\int \limits_{\Omega}{  \tfrac{ v^2 }{F(x)} \,dx} \leqslant C_P(R_0) \int \limits_{\Omega}{\omega(x) \Big(\frac{\partial v}{\partial x}\Big)^2 \,dx}
\text{ with } v (1) = 0,
$$
where
$$
C_P(R_0) =  \int \limits_{\Omega}{  \tfrac{1}{F(x)} \Bigl( \int \limits_{x}^1 { \tfrac{dy}{\omega(y)} }  \Bigr)   \,dx}.
$$
Hence,  we end up with  relations
\begin{multline}\label{ch-3-3}
\int \limits_{\Omega}{F(x) \bigg(\frac{\partial z}{\partial x}\bigg)^2   \,dx} \leqslant e^{-  \frac{t}{N C_P(R_0)}}\int \limits_{\Omega}{F(x) \bigg(\frac{\partial z_{0}}{\partial x}\bigg)^2 \,dx}   \to 0
\text{ as } t \to +\infty .
\end{multline}
As a result,  we obtain the convergence
$$
z(x,t) \to 0 \quad\text{ strongly  in }\quad W^{1,2}(\Omega) \text{ as } t \to +\infty.
$$
if only the inequality holds
$$
\int \limits_{\Omega}{ \bigg( \omega^{-1}(x) z_0^2(x) + F(x) \bigg(\frac{\partial z_{0}}{\partial x}\bigg)^2 \bigg) \,dx} < +\infty.
$$
The simple consequence of this fact and convergence 
 (\ref{ch-3-3}) is the bound
$$
z(x,t) \leqslant  x^{\frac{1}{2}} e^{-  \frac{t}{2N C_P(R_0)}}\Bigl(\int \limits_{\Omega}{F(x) \bigg(\frac{\partial z_{0}}{\partial x}\bigg)^2 \,dx}\Bigr)^{\frac{1}{2}},
$$
which in turn provides the desired relation
\begin{equation*}\label{asym-1}
p(x,t) \leqslant \tfrac{x^{\frac{1}{2}} F(x)}{f(x)} e^{-  \frac{t}{2N C_P(R_0)}}\Bigl(\int \limits_{\Omega}{F(x) \bigg(\frac{\partial z_{0}}{\partial x}\bigg)^2 \,dx}\Bigr)^{\frac{1}{2}}.
\end{equation*}

At this point we prove statement \textbf{(ii)} of the theorem.  Multiplying (\ref{ch-1}) by $\omega(x) \psi(x)   z(x,t) $ (as for the test function $\psi$ see Definition \ref{d.bis1}) and integrating  over
$\Omega$, we obtain
\begin{align*}\label{ch-911}\notag
&\frac{1}{2} \frac{d}{dt}\int \limits_{\Omega}{
\psi(x) z^2  \,dx}  + \frac{1}{ 2N} \int \limits_{\Omega}{
f(x) \psi(x) \bigg(\frac{\partial z}{\partial x}\bigg)^2 \,dx} \\
& =    
\frac{1}{2N} \bigg( f(x) \psi(x) z \frac{\partial z}{\partial x}  - \frac{ 1}{2} ( \omega (x) \psi(x))' F(x)z^2 \bigg)  \bigg |_0^1\\
&+
\frac{1}{ 4N} \int \limits_{\Omega}z^2\frac{d}{dx}
\bigg(F(x)\frac{d}{dx}\bigg( \omega (x) \psi(x)\bigg) \bigg)  \,dx   .
\end{align*}
Then choosing here 
$$
\psi(x) = \omega^{-1}(x) \int \limits_{0}^x{ \tfrac{dy}{F(y)}}  = \tfrac{F(x)}{f(x)} \int \limits_{0}^x{ \tfrac{dy}{F(y)}}
\to \tfrac{1}{1+R_0} \text{ as } x \to 0,
$$
we arrive at the equality 
\begin{equation}\label{ch-5}
\frac{d}{dt}\int \limits_{\Omega}{
\psi(x) z^2  \,dx} + \frac{1}{N}\int \limits_{\Omega}{
f(x) \psi(x) \bigg(\frac{\partial z}{\partial x}\bigg)^2 \,dx} + \frac{1}{2N} z^2(1,t) = 0.
\end{equation}
In order to manage the left hand-side of this equality, we apply the Hardy inequality
\begin{equation*}\label{hardy-3}
\int \limits_{\Omega}{
\psi(x) v^2  \,dx} \leqslant C_H(R_0) \int \limits_{\Omega}{
f(x)\psi(x) v^2_{x} \,dx} \text{ with } v(0) =0,
\end{equation*}
where $A(R_0) \leqslant C_H(R_0) \leqslant 4 A(R_0)$ and
$$
A(R_0) =    \mathop {\sup} \limits_{r \in (0,1)}  \Bigl(\int \limits_{0}^r {
\tfrac{dx}{f(x)\psi(x) } }  \Bigr) \Bigl(\int \limits_{r}^1 {
\psi(x) \, dx }  \Bigr).
$$
Thus, we easily conclude that
\begin{equation}\label{ch-6}
 \int \limits_{\Omega}{
\psi(x) z^2(x,t)  \,dx} \leqslant e^{- \frac{t}{  N C_H(R_0)} } \int \limits_{\Omega}{
\psi(x) z_0^2(x)  \,dx} \to 0 \text{ as } t \to +\infty.
\end{equation}
Note that as $\psi(x)$ is bounded from bottom then we have  $z \to 0$ in $L^2$ as $t \to +\infty$.

Now, multiplying the equation in (\ref{ch-1}) by $- \omega(x)\phi(x) \frac{\partial}{\partial x}\big(F(x) \frac{\partial z}{\partial x}\big)$ and integrating  over
$\Omega$, we obtain
\begin{align*}
&\frac{1}{2} \frac{d}{dt}\int \limits_{\Omega}{
\phi(x) F(x) \bigg(\frac{\partial z}{\partial x}\bigg)^2  \,dx}  + \frac{1}{2N} \int \limits_{\Omega}
\omega(x)\phi(x) \bigg(\frac{\partial}{\partial x}\bigg(F(x)\frac{\partial z}{\partial x}\bigg)\bigg)^2 \,dx\\& = 
\bigg( \phi(x) F(x) \frac{\partial z}{\partial t} \frac{\partial z}{\partial x}  - \frac{1}{4N} \omega(x) \phi'(x)F^2(x) \bigg(\frac{\partial z}{\partial x}\bigg)^2 \bigg)   \bigg |_0^1 \\
& +
\frac{1}{4N} \int \limits_{\Omega}
(\omega(x) \phi'(x))'  F^2(x)\bigg(\frac{\partial z}{\partial x}\bigg)^2   \,dx   .
\end{align*}
Taking $\phi(x)$ such that $(\omega(x) \phi'(x))'  F^2(x) = 2 f(x) \psi(x)$, i.\,e.
$$
\phi (x) = 2 \int \limits_{0}^x { \tfrac{1}{\omega (y)} \Bigl( \int \limits_{0}^y { \tfrac{1}{F(v)} \Bigl( \int \limits_{0}^v {
\tfrac{ds}{F(s)} } \Bigr) dv}  \Bigr)  dy}  \sim \tfrac{x^2}{2(R_0+1)} \text{ as } x \to 0,
$$
we have
\begin{align*}
 & \frac{d}{dt}\int \limits_{\Omega}
\phi(x) F(x) \bigg(\frac{\partial z}{\partial x}\bigg)^2  \,dx  + \frac{1}{ N} \int \limits_{\Omega}
\omega(x)\phi(x) \bigg(\frac{\partial}{\partial x}\bigg(F(x)\frac{\partial z}{\partial x}\bigg)\bigg)^2 \,dx
\\
& = \frac{1}{N} \int \limits_{\Omega}
f(x) \psi(x) \bigg(\frac{\partial z}{\partial x}\bigg)^2  \,dx.
\end{align*}
Collecting this equality with (\ref{ch-5}),  we arrive at
\begin{align}\label{ch-8}\notag
  &\frac{d}{dt}\int \limits_{\Omega}
\bigg( \phi(x) F(x) \bigg(\frac{\partial z}{\partial x}\bigg)^2+ \psi(x) z^2 \bigg) \,dx \\ \notag
& + \frac{1}{ N} \int \limits_{\Omega}
\omega(x)\phi(x) \bigg(\frac{\partial }{\partial x}\bigg(F(x)\frac{\partial z}{\partial x}\bigg)\bigg)^2 \,dx 
\\&
+ 
  \frac{1}{2N} z^2(1,t) =0.
\end{align}
Then, applying the estimate
$$
\int \limits_{\Omega} \frac{\phi(x)}{F(x)} v ^2  \,dx \leqslant C_P(R_0) \int \limits_{\Omega}\omega(x)\phi(x) \bigg(\frac{\partial v}{\partial x}\bigg)^2 \,dx
\text{ with } v(1) = 0,
$$
where
$$
C_P(R_0) =  \int \limits_{\Omega}{  \tfrac{\phi(x)}{F(x)} \Bigl( \int \limits_{x}^1 { \tfrac{dy}{\omega(y)\phi(y)} }  \Bigr)   \,dx},
$$
to (\ref{ch-8}), we conclude
\begin{align*}\label{ch-9}
&\int \limits_{\Omega}\phi(x) F(x) \bigg(\frac{\partial z}{\partial x}\bigg)^2   \,dx \leqslant
\Bigl[ \int \limits_{\Omega}{\phi(x) F(x) \bigg(\frac{\partial z_{0}}{\partial x}\bigg)^2 \,dx} \\&+
 \tfrac{1}{NC_P(R_0)} \bigg(1 - \tfrac{NC_P(R_0)C_H(R_0)}{C_H(R_0) -C_P(R_0)} +
 \tfrac{NC_P(R_0)C_H(R_0)}{C_H(R_0) -C_P(R_0)}
  e^{\tfrac{C_H(R_0) - C_P(R_0)}{NC_P(R_0)C_H(R_0)}t} \bigg)\\&
	\times\int \limits_{\Omega}\psi(x) z_0^2(x)  \,dx \Bigr] e^{- \frac{t}{NC_P(R_0)}}
	\to 0 \quad 
\text{ as } \quad t \to +\infty .
\end{align*}
As a result, 
$$
z(x,t) \to 0  \text{ and }
x \frac{\partial z}{\partial x}(x,t) \to 0  \text{ strongly in } L^2(\Omega) \text{ as } t \to +\infty,
$$
if only the inequality holds
$$
\int \limits_{\Omega}{ \bigg( \psi(x) z_0^2(x) +  \phi(x)F(x) \bigg(\frac{\partial z_{0}}{\partial x}\bigg)^2 \bigg) \,dx} < +\infty.
$$
Finally, the passage to $p$ in these relations completes the proof of assertion (ii) and, as a consequence, of  Theorem \ref{Th-d}.

\end{proof}



\section{Particular solutions}
In this section, we will illustrate the result of Theorem~\ref{Th-d} by  particular solutions. First, we analyze classical solutions to problem  (\ref{ch-1}) and then we discuss particular weak solutions.
\subsection{Particular classical solutions}
Introducing the new variable
$$
s = \sqrt{2N}\int \limits_0^x {\tfrac{dy}{f^{\frac{1}{2}}(y)}},
$$
and denoting 
$$
l(s) := \sqrt{2N} \tfrac{g(x)}{f^{\frac{1}{2}}(x)} = \sqrt{\tfrac{2N}{R_0}} \,\tfrac{\sin \bigl( \frac{1}{2}\sqrt{\frac{R_0}{2N}} s \bigr)
\bigl[R_0 -1 -(R_0+1) \sin^2 \bigl( \frac{1}{2}\sqrt{\frac{R_0}{2N}} s \bigr) \bigr] }{|\cos\bigl( \frac{1}{2}\sqrt{\frac{R_0}{2N}} s \bigr)|},
$$
$$
s_1 := 2 \sqrt{\tfrac{2N}{R_0}} \arcsin \bigl( \sqrt{\tfrac{R_0 }{R_0+1}} \bigr),
$$
we rewrite problem (\ref{ch-1}) in the form
\begin{equation}\label{asy-0}
\begin{cases}
\frac{\partial z}{\partial t} =  \frac{\partial^{2} z}{\partial s^{2}} + l(s) \frac{\partial z}{\partial s},\qquad s\in(0,s_1),\quad t\in(0,T),\\
z(0,t) =0, \quad \frac{\partial z}{\partial s}(s_1,t) = 0,\quad t\in[0,T].
\end{cases}
\end{equation}
It is worth noting that to reach \eqref{asy-0}, we incorporate the following simple verified relations:
$$
s =
\begin{cases}
 2 \sqrt{\tfrac{2N}{R_0}} \arcsin \bigl( \sqrt{\tfrac{R_0 }{R_0+1}} x^{\frac{1}{2}} \bigr),\quad 
\text{ if }\quad R_0 > 0,\\
\,
\\
2 \sqrt{2N} x^{\frac{1}{2}},\qquad \quad\text{ if }\quad R_0 = 0, 
\end{cases}
$$
or as consequence
$$
x =
\begin{cases}
\tfrac{R_0+1}{R_0} \sin^2 \bigl( \tfrac{1}{2}\sqrt{\tfrac{R_0}{2N}} s \bigr), 
 \quad
\text{ if } \quad R_0 > 0,  \\
\,
\\
\tfrac{1}{8N} s^2,\qquad \qquad\text{ if } \quad R_0 = 0. 
\end{cases}
$$
Separating variables in (\ref{asy-0}): 
$$
z(s,t)= T(t) S(s),
$$
leads to the problems
$$
\tfrac{T'(t)}{T(t)} = \tfrac{S''(s) + l(s) S'(s)}{S(s)} = - \lambda,
$$
whence
$$
T'(t) = - \lambda T(t),
$$
\begin{equation}\label{asy-0-1}
S''(s) + l(s) S'(s) =  - \lambda S(s)
\end{equation}
with
$$
S(0) =0, \ S'(s_1) = 0.
$$
After that, multiplying (\ref{asy-0-1}) by $p(s) := e^{ \int \limits_0^s{ l(y)\,dy}}$,  we immediately obtain the equation
$$
- (p(s)S'(s))' = \lambda p(s) S(s).
$$
Then setting 
$$U(s) = p^{\frac{1}{2}}(s) S(s)\quad q(s) = \tfrac{(p^{\frac{1}{2}}(s))''}{p^{\frac{1}{2}}(s)} = \tfrac{1}{2} \bigl( l'(s) + \tfrac{1}{2} l^2(s) \bigr), 
$$ 
we arrive at the classical Sturm-Liouville problem with the continuous potential $q(s)$
\begin{equation}\label{egen}
\begin{cases}
- U''(s) + q(s) U(s) = \lambda U(s),\qquad s\in(0,s_1), \\
 U(0) = 0, \quad U'(s_1) = 0.
\end{cases}
\end{equation}
The standard technical calculations provide the following asymptotic  behavior of
eigenvalues and eigenfunctions  to problem 
\eqref{egen}
$$
\lambda_k \sim (\tfrac{\pi}{s_1} )^2 \bigl(k + \tfrac{1}{2} \bigr)^2, \ \ U_k(s) \sim \sin \bigl( \tfrac{\pi}{s_1} ( k + \tfrac{1}{2}) s  \bigr),
$$
or returning to  (\ref{asy-0-1}) 
$$
\lambda_k \sim (\tfrac{\pi}{s_1} )^2 \bigl(k + \tfrac{1}{2} \bigr)^2, \ \
S_k(s) \sim e^{-\frac{1}{2} \int \limits_0^s{ l(y)\,dy}} \sin \bigl( \tfrac{\pi}{s_1} ( k + \tfrac{1}{2}) s  \bigr).
$$
Thus, problem (\ref{asy-0}) has a particular solution
$$
z(s,t) = \mathop {\sum} \limits_{k=0}^{+\infty} {c_k e^{-\lambda_k t} S_k(s)},
$$
which in turn means
$$
z(x,t) = \mathop {\sum} \limits_{k=0}^{+\infty} {c_k e^{-\lambda_k t} \varphi_k(x)},
$$
where
$$
\lambda_k \sim  \tfrac{ \pi^2 }{N}\bigl(k + \tfrac{1}{2} \bigr)^2, \ \
\varphi_k(x) \sim e^{- N^{\frac{3}{2}} x} \sin \Bigl(  \pi ( k + \tfrac{1}{2}) \tfrac{\arcsin \bigl( \sqrt{\frac{R_0 }{R_0+1}} x^{\frac{1}{2}} \bigr)}{\arcsin \bigl( \sqrt{\frac{R_0 }{R_0+1}}   \bigr)}  \Bigr).
$$

Finally, keeping in mind the relation $z(x,t) = \omega(x) p(x,t)$, 
we deduce the particular classical solution
$$
p(x,t) = \tfrac{1}{\omega(x)} \mathop {\sum} \limits_{k=0}^{+\infty} {c_k e^{-\lambda_k t} \varphi_k(x)}.
$$
It is worth noting that the asymptotic behavior of the solution $\frac{C_1}{ \sqrt{x}e^{ C_2 t}}$ as $ x \to 0^+$ has a good agreement
with Theorem~\ref{Th-d} (i).


\subsection{Particular weak solutions}
\begin{figure}
     \includegraphics[width = 0.49\textwidth]{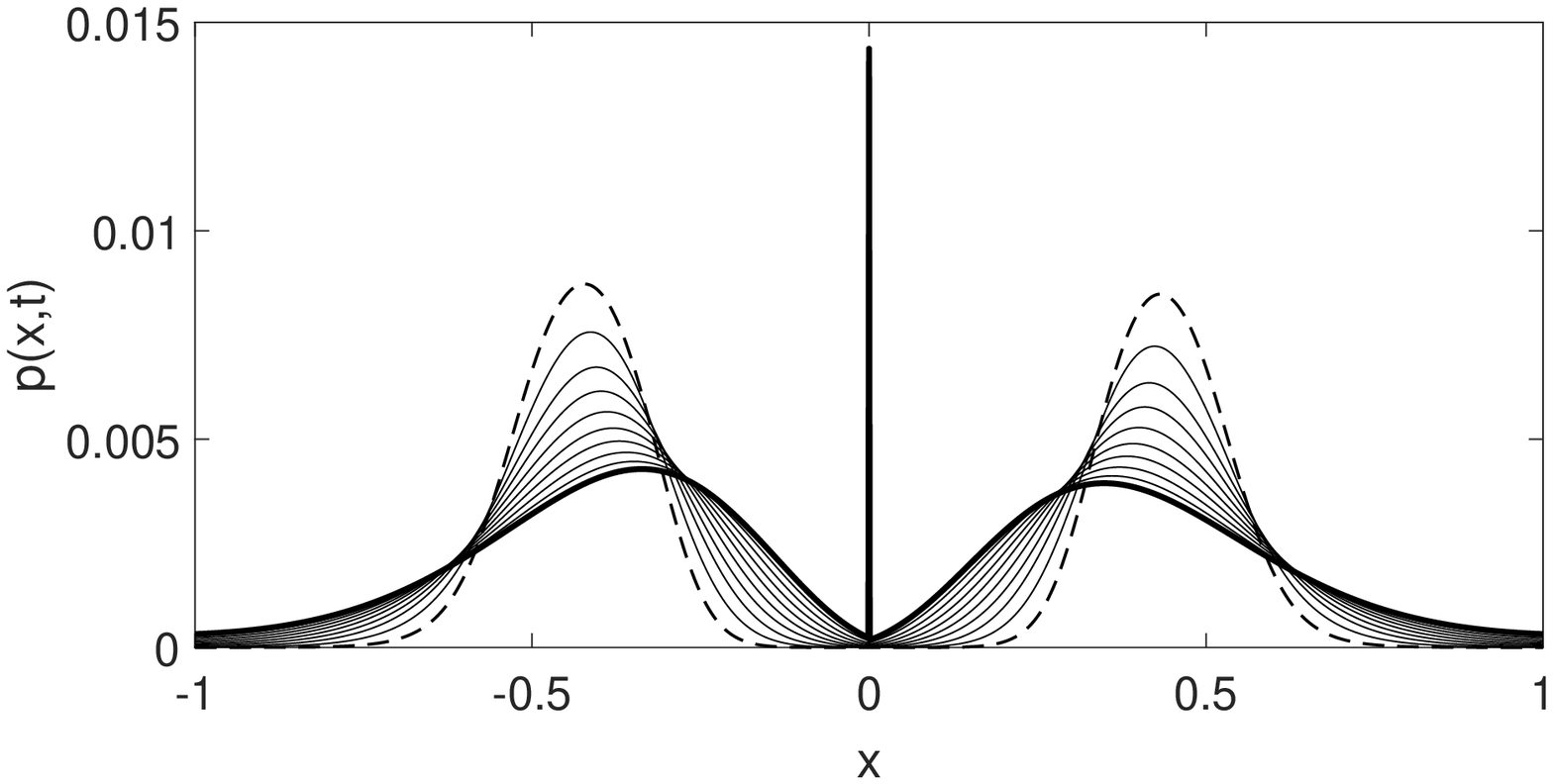}
    \includegraphics[width = 0.49\textwidth]{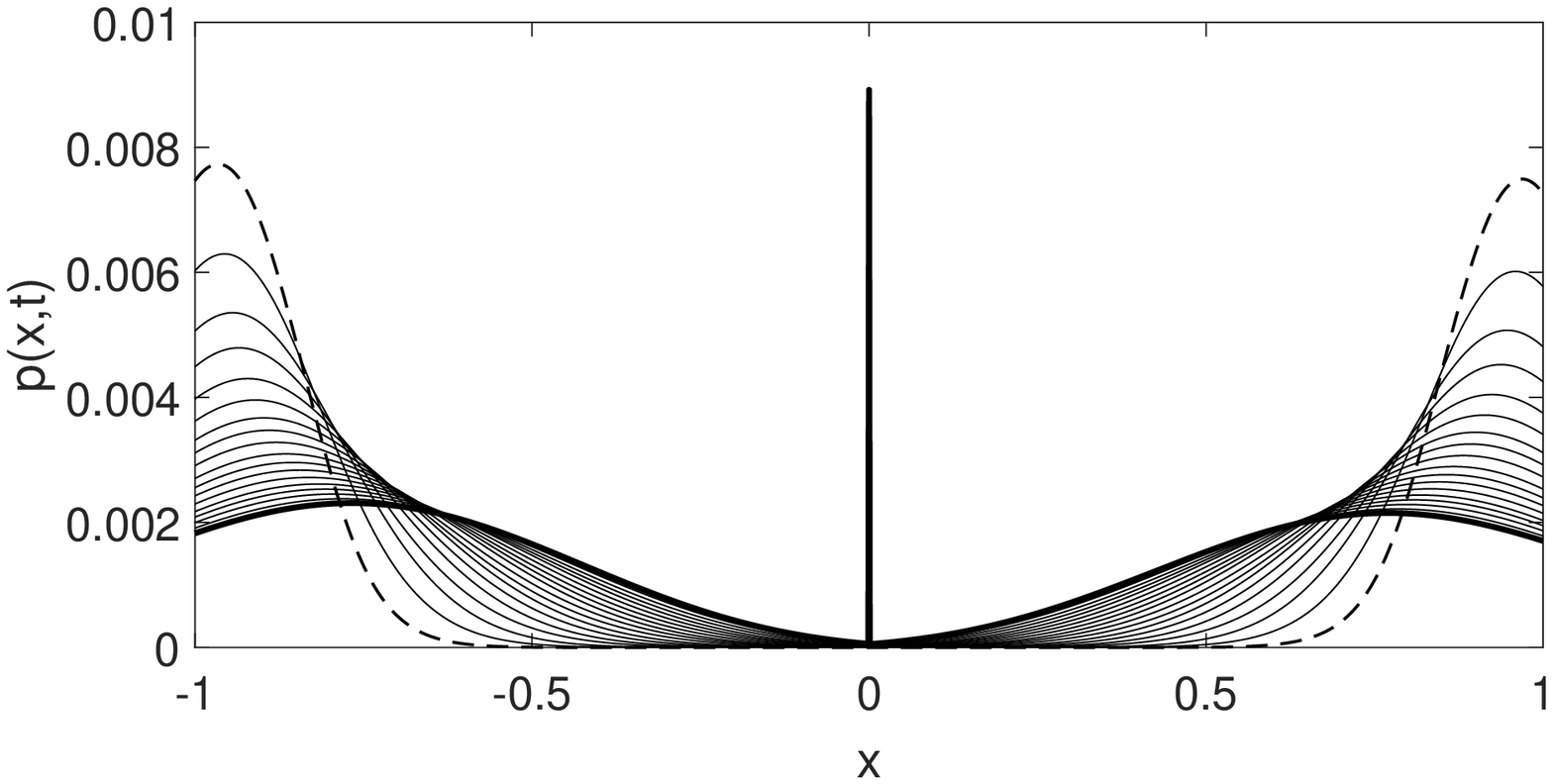}
    \caption{These two pictures illustrate existence of delta function type solution for a symmetrised problems with  $R_0 = 0$ and $N = 10$.}
    \label{steadyr7a}
\end{figure}
First of all we remark that
with zero on the boundary the integral $\int_{0}^{a} f(z) \delta(z) dz$ is \textit{a priori} not well defined. Then, we denote positive and non-negative cut of functions by $s1_{x>0}$ and $s1_{x\geq 0}$,respectively.
 It corresponds to integrating $\delta$ function against the function $f 1_{x>0}$ (or possibly $f 1_{x \geq 0}$ ) which is not continuous at the origin $x=0$, where the support of delta function lies. Some formal not analytically justified expressions have been used in a technical literature like $\int_{0}^{a} f(z) \delta(z) d z=\int_{-a}^{0} f(z) \delta(z) d z=\frac{1}{2} f(0)$. This is the reason why to justify a delta function type solution for our problem we will use a symmetrization method by considering a problem of extended domain $[-1,1]$.
Now we look for a solution to a symmetrically extended problem (\ref{c-1})--(\ref{c-3}) on the interval $(-1,1)$ in the form $p(x,t) = \eta(t) \delta_0(x)$.

Multiplying symmetrized equation (\ref{c-1}) by $\phi (x) \in C^2[-1,1]$ with compact support   and $\phi(0) \neq 0$,
after integrating by parts in $Q_T := (-1,1) \times (0,T)$, we have
$$
\iint \limits_{Q_T} { \frac{\partial p}{\partial t} \phi(x) \,dx dt} =
\frac{1}{2N} \iint \limits_{Q_T} { \bigl( \tilde{f}(x) p \phi''(x) + 2N \tilde{g}(x) p \phi'(x) \bigr) \,dx dt},
$$
where $\tilde{f}$ and $\tilde{g}$ are even continuation of $f$ and $g$, accordingly.
Taking $p(x,t) = \eta(t) \delta_0(x)$ in the last equality, we deduce that
$$
 ( \eta (T) -\eta(0) ) \phi(0)   =   ( \tfrac{1 }{2N} f(0)\phi''(0) + g(0) \phi'(0) ) \int \limits_{0}^T {\eta(t) \,dt} = 0.
$$
Thanks to the inequality $\phi(0) \neq 0$,  we have
$$
\eta (T)  = \eta(0) = M > 0  .
$$
As a result, symmetrized equation (\ref{c-1}) has the following solution:
$$
p (x,t) =  M \delta_0 (x) \text{ for all } (x,t) \in (-1,1) \times (0,+\infty)  .
$$

Convergence of solution toward delta function is illustrated in Figure {\ref{steadyr7a}}. It is interesting to mention that a non-smooth change of variables $y = 2\,\sqrt{x}$ (for the case $R_0 = 0$) will remove the degeneracy from the equation but the whole long-time dynamics will not be recovered as a in terms of $y$ as a global attractor type solution $C e^t$ that satisfied a no-flux boundary conditions in terms of variable $y$ will not be satisfying no-flux boundary conditions in terms of variable $x$. Although the $C e^t$ solves the original problem with Neumann boundary conditions (which make the original problem ill-posed) it is unstable and a slight perturbation will drive the dynamics toward delta function.

\section{Local classical solvability and asymptotic behavior of classical solutions near $x=0$}
 Here, we will discuss existence and behavior of a classical solution near the degenerate point $x=0$. Besides, we will analyze sufficient conditions on the given data in the model which provides the fulfillment of the conservation law. To this end, we consider the initial-boundary value problem to more general linear degenerate equation and the boundary condition than in the original problem (\ref{c-1})--(\ref{c-3}).

Let
$
\Omega=(0,l)\subseteq\R_{+}=\{x:x>0\} $
be a segment with a boundary $\partial\Omega=\{x=0\}\cup\{x=l\}$. In our further consideration, we also use the notation $\Gamma_1$ for the right point of the boundary, i.e. $\Gamma_1=\{x=l\}$.
For an arbitrary fixed time $T>0$, we denote
\[
\partial\Omega_{T}=\partial\Omega\times[0,T]\quad\text{and}\quad \Gamma_{1,T}=\Gamma_1\times[0,T].
\]
We analyze the linear degenerate equation in the unknown function $u=u(x,t):\Omega_{T}\to\R,$
\begin{equation}\label{1.1}
\frac{\partial u}{\partial t}-\mathcal{L}u=f(x,t)
\end{equation}
supplemented with the initial condition
\begin{equation}\label{1.2}
u(x,0)=u_{0}(x)\quad\text{in}\quad\bar{\Omega},
\end{equation}
subject to the condition of the third kind (\textbf{3BC}):
\begin{equation}\label{1.3}
\mathcal{M}u=\psi(x,t)\quad\text{on}\quad\Gamma_{1,T},
\end{equation}
where the functions $f,u_0$ and $\psi$ are prescribed.

Coming to the operator involved, $\mathcal{L}$ is the linear degenerate elliptic operator of the second order with time-dependent coefficients, namely,
\[
\mathcal{L}=x a_{0}(x,t)\frac{\partial^{2}}{\partial x^{2}}+a_1(x,t)\frac{\partial}{\partial x}+a_2(x,t),
\]
while the operator $\mathcal{M}$ reads as
\[
\mathcal{M}=b_1(x,t)\frac{\partial}{\partial x}+b_{2}(x,t).
\]


\subsection{Solvability of \eqref{1.1}-\eqref{1.3} in smooth classes and behavior of the solution at $x=0$}
\label{s3}
\noindent
We start by introducing some general assumptions:
\begin{description}
    \item[H1 (Conditions on the coefficients):] There exist positive $\delta_{0}$ and $\delta_{1}$ such that
		\begin{equation}\label{3.1}
		a_0(x,t)\geq\delta_0\,\text{ for any }(x,t)\in\bar{\Omega}_{T}, \,
		a_1(0,t)\geq\delta_1\text{ for any } t\in[0,T].
		\end{equation}
	Besides,
	\[
	a_{0},a_{1},a_{2}\in\C_{\alpha}^{\alpha,\alpha/2}(\bar{\Omega}_{T}),\, b_{1},b_{2}\in\C^{1+\alpha,\frac{1+\alpha}{2}}(\Gamma_{1,T}),\]
	\[
	b_{1}(l,t)\neq 0\quad \text{for all } t\in[0,T].
	\]
	  \item[H2 (Conditions on the given functions):]
		\[
		u_{0}(x)\in\C_{\alpha}^{2+\alpha}(\bar{\Omega}),\quad f\in\C_{\alpha}^{\alpha,\alpha/2}(\bar{\Omega}_{T}),\quad \psi\in\C^{1+\alpha,\frac{1+\alpha}{2}}(\Gamma_{1,T}).
		\]
		\item[H3 (Compatibility conditions):] When $x\in\Gamma_{1}$, the compatibility condition at $t=0$ holds
		\[
		\mathcal{M}u_{0}(x)|_{t=0}=\psi(x,0).
		\]
		\end{description}	
\begin{remark}\label{r3.1}
It is worth mentioning that assumption \eqref{3.1} on the coefficient $a_1$ and Fiker condition (see. i.e. \cite[(1.13), p.17]{OR})
guarantee well-posedness in smooth classes of problem \eqref{1.1}-\eqref{1.3} (i.e. the initial-boundary value problem is stated without a boundary condition at the point $x=0$).
\end{remark}
We are now in the position to state our main result concerning the solvability of \eqref{1.1}-\eqref{1.3}.
\begin{theorem}\label{t3.1}
Let $T>0$ be arbitrarily fixed and let assumptions \textbf{H1-H3} hold. Then problem \eqref{1.1}-\eqref{1.3} admits a unique classical
solution $u=u(x,t)$ on the space-time rectangle $\Omega_{T}$, satisfying regularity $u\in\C_{\alpha}^{2+\alpha,\frac{2+\alpha}{2}}(\bar{\Omega}_{T})$. Besides, the following estimate holds
\begin{align}\label{3.2}\notag
\|u\|_{\C_{\alpha}^{2+\alpha,\frac{2+\alpha}{2}}(\bar{\Omega}_{T})}&\leq C_{1}[\|u_{0}\|_{\C_{\alpha}^{2+\alpha}(\bar{\Omega})}
+\|f\|_{\C_{\alpha}^{\alpha,\alpha/2}(\bar{\Omega}_{T})}
+\|\psi\|_{\C^{1+\alpha,\frac{1+\alpha}{2}}}
] \\ &:= 
C_{1}\mathfrak{F}(u_0,f,\psi)
\end{align}
with some $C_1>0$ being independent of the right-hand sides of \eqref{1.1}-\eqref{1.3}.

If, in addition, the following assumption holds.

\noindent\textbf{H4:} There is a real $s_0$ such that
\begin{equation}\label{3.3}
\begin{cases}
s_0>\alpha,\\
\quad 2s_0a_{0}(0,t)+a_{1}(0,t)>0,\\
 \underset{x\to 0}{\lim}\bigg|\frac{a_1(x,t)+(s_0-1)a_{0}(x,t)}{x}\bigg|<C,
\end{cases}
\end{equation}
for all $t\in[0,T]$ and
\[
u_{0}\in E_{s_0}^{2+\alpha}(\bar{\Omega}),\quad f\in E_{s_0}^{\alpha,\alpha/2}(\bar{\Omega}_{T}).
\]
Then, the classical solution $u$ satisfies the regularity 
$$(x^{-s_{0}}u)\in\C_{\alpha}^{2+\alpha,\frac{2+\alpha}{2}}(\bar{\Omega}_{T}),$$
 and
\begin{equation}\label{3.4}
\|x^{-s_{0}}u\|_{\C_{\alpha}^{2+\alpha,\frac{2+\alpha}{2}}(\bar{\Omega}_{T})}\leq C_{2}[\|u_{0}\|_{E_{s_0}^{2+\alpha}(\bar{\Omega})}
+\|f\|_{E_{s_0}^{\alpha,\alpha/2}(\bar{\Omega}_{T})}
+\mathfrak{F}(u_0,f,\psi)
].
\end{equation}
\end{theorem}
\begin{remark}\label{r3.2}
It is apparent that estimate \eqref{3.4} where $s_0>\alpha$ provides the equality
\begin{equation}\label{3.5}
u(0,t)=0\quad\text{for all}\quad t\in[0,T].
\end{equation}
\end{remark}
\begin{example}\label{e3.1}
The following is an example of coefficients $a_0,a_1$ and the weight $s_0$ satisfying condition \eqref{3.3} in the case of
$\alpha\in(0,1/2)$:
\[
s_0=\frac{1}{2}\quad\text{and}\quad a_1(x,t)=\frac{a_0(x,t)}{2},\quad (x,t)\in\bar{\Omega}_{T},
\]
where $a_0$ meets requirement \textbf{H1}.
\end{example}
It is easy to verify that condition \eqref{3.3} is not satisfied in the case of original problem (\ref{c-1})--(\ref{c-3}). That means, estimate \eqref{3.4} does not provid equality \eqref{3.5}. Nevertheless, our next result states this equality under weaker assumptions on the coefficients than \eqref{3.3}. Namely, we will assume additionally to \textbf{H1-H3} the following conditions.
\begin{description}
 \item[H5] Let  $u_0\in E_{s}^{2+\alpha}(\bar{\Omega}),$ $f\in E_{s-1}^{\alpha,\alpha/2}(\bar{\Omega}_{T})$ with $s>1+\alpha$.
\item[H6] Let  $a_0,a_{1}\in\C([0,T],\C^{1}(\bar{\Omega}))$. Besides, there is $\delta_{2}>0$ and the function $a(x,t)\in\C([0,T],\C^{1}(\bar{\Omega}))\cap\C_{\alpha/2}^{\alpha,\alpha/2}(\bar{\Omega}_{T})$ (including the case $a\equiv 0$) satisfying the inequality
\[
\underset{\bar{\Omega}_{T}}{\sup}\frac{|a(x,t)|}{x}\leq C,
\]
such that there holds
\[
a_{1}(x,t)-a_{0}(x,t)=\delta_{2}+a(x,t).
\]
\end{description}
\begin{theorem}\label{t3.2}
Let $\psi\in\C^{1+\alpha,\frac{1+\alpha}{2}}(\Gamma_{1,T})$. Under assumptions \textbf{H1, H3, H5} and \textbf{H6}, classical solution \eqref{1.1}-\eqref{1.3} vanishes at $x=0$ for all $t\in[0,T]$, i.e. equality \eqref{3.5} holds.
\end{theorem}
The rest part of this section is devoted to the proof of Theorems \ref{t3.1} and \ref{t3.2}.

\subsection{Some technical results}
\label{s3.2}
\noindent In this section we provide some technical results which are key points in the proof of Theorem \ref{t3.1} and as a consequence of Theorem \ref{t3.2}.
\subsubsection{Some auxiliary estimates}
\label{s3.2.1}
\noindent
First we describe some properties of the solution to the initial-value problem for the degenerate linear equation, which will be key point in the proof of Theorem \ref{t3.1} in Section \ref{s3.3}.

Let the function $v_{1}=v_{1}(x,t)$ solves the Cauchy problem
\begin{equation}\label{3.6}
\begin{cases}
\frac{\partial v_{1}}{\partial t}-x\frac{\partial^{2}v_1}{\partial x^{2}}-A_{1}\frac{\partial v_{1}}{\partial x}=F(x,t) \quad \text{in}\quad\R_{+,T}=\R_{+}\times(0,T),\\
v_{1}(x,0)=v_{1,0}(x),\quad\text{in}\quad\R_{+},
\end{cases}
\end{equation}
where $F$ and $v_{1,0}$ are some given functions, and $A_1$ is a given positive number.

The classical solvability of problem \eqref{3.6} was studied in \cite[\S2-3]{BK}. In particularly, the following result, stated as a lemma, subsumes Theorem 3.1 and estimate \eqref{3.2} in \cite{BK}.
\begin{lemma}\label{l3.1}
Let $A_{1}$ and $v_{1,0}\in\C_{\alpha}^{2+\alpha}(\bar{\R}_{+})$, $F\in\C_{\alpha}^{\alpha,\alpha/2}(\bar{\R}_{+,T})$. Besides, we assume that there exists a positive number $r_0$ such that
\[
v_{1,0}(x)=F(x,t)=0\quad \text{if}\quad x>r_0.
\]
Then, for any fixed $T>0$, there is a unique classical solution $v_{1}(x,t)$ to problem \eqref{3.6}. In addition, the following estimate holds:
\[
\|v_{1}\|_{\C_{\alpha}^{2+\alpha,\frac{2+\alpha}{2}}(\bar{\R}_{+,T})}\leq C[\|F\|_{\C_{\alpha}^{\alpha,\alpha/2}(\bar{\R}_{+,T})}+
\|v_{1,0}\|_{\C_{\alpha}^{2+\alpha}(\bar{R}_{+})}].
\]
Here, the generic constant $C$ is independent of the right-hand side in \eqref{3.6} and depends only on $A_1$ and $T$.
\end{lemma}
The next lemma allows us to reduce \eqref{1.1}-\eqref{1.3} to problem with homogenous initial data.
\begin{lemma}\label{l3.2}
There exists a universal constant $C^{*}>0$ with the following property: for any functions $\mathcal{V}_{0}=\mathcal{V}_{0}(x)\in\C_{\alpha}^{2+\alpha}(\bar{\Omega})$ and $\mathcal{V}_{1}=\mathcal{V}_{1}(x)\in\C_{\alpha}^{\alpha}(\bar{\Omega})$, there exists a function $\mathcal{V}=\mathcal{V}(x,t)\in\C_{\alpha}^{2+\alpha,\frac{2+\alpha}{2}}(\bar{\Omega}_{T})$ such that
\[
\mathcal{V}(x,0)=\mathcal{V}_{0}(x),\quad \frac{\partial\mathcal{V}}{\partial t}(x,0)=\mathcal{V}_{1}(x),
\]
and
\[
\|\mathcal{V}\|_{\C_{\alpha}^{2+\alpha,\frac{2+\alpha}{2}}(\bar{\Omega}_{T})}\leq C(\|\mathcal{V}_{0}\|_{\C_{\alpha}^{2+\alpha}(\bar{\Omega})}
+
\|\mathcal{V}_{1}\|_{\C_{\alpha}^{\alpha}(\bar{\Omega})}).
\]
\end{lemma}
\begin{proof}
Denoting
$$\mathcal{F}(x):=\mathcal{V}_{1}(x)-x\frac{\partial^{2} \mathcal{V}_{0}}{\partial x^{2}}(x)-A_1\frac{\partial\mathcal{V}_0}{\partial x},$$
we build (see \cite[Theorem 4.1]{LSU})
extensions $\bar{\mathcal{V}}_{0}$ and $\bar{\mathcal{F}}$ of
the functions $\mathcal{V}_{0}$ and $\mathcal{F}$ on $\R_{+}$ such that
\begin{equation}
\label{3.7}
\begin{aligned}
\|\bar{\mathcal{V}}_{0}\|_{\C_{\alpha}^{2+\alpha}(\bar{\R}_{+})} &\leq
C\|\mathcal{V}_{0}\|_{\C_{\alpha}^{2+\alpha}(\bar{\Omega})},\\
\|\bar{\mathcal{F}}\|_{\C_{\alpha}^{\alpha}(\bar{\R}_{+})} & \leq C
\|\mathcal{F}\|_{\C_{\alpha}^{\alpha}(\bar{\Omega})}\leq
C(\|\mathcal{V}_{0}\|_{\C_{\alpha}^{2+\alpha}(\bar{\Omega})}+\|\mathcal{V}_{1}\|_{\C_{\alpha}^{\alpha}(\bar{\Omega})}),
\end{aligned}
\end{equation}
and, besides,  the functions  $\bar{\mathcal{V}}_{0}$ and $\bar{\mathcal{F}}$ have compact
supports.
After that, we define the function $\mathcal{V}(x,t)$ to be the solution to the Cauchy problem
$$
\begin{cases}
\frac{\partial \mathcal{V}}{\partial t}-x\frac{\partial^{2}\mathcal{V}}{\partial x^{2}}-A_1\frac{\partial \mathcal{V}}{\partial x}=
\bar{\mathcal{F}}(x)\quad\text{in }
\R_{+,T},
\\
\mathcal{V}(x,0)=\bar{\mathcal{V}}_{0}(x).
\end{cases}
$$
Applying Lemma \ref{l3.1} to this problem, and taking into
account \eqref{3.7}, the claim is proven.
\end{proof}


\subsubsection{Domains. Some auxiliary propositions. }
\label{s3.2.2}
\noindent
In order to prove the first part of  Theorem \ref{t3.1}, we will construct a
\textit{regularizer} (see \cite[Section 4]{LSU}). To this end, we need a
special covering of the domain $\Omega$. We take two collections
of open sets $\{\varpi^{m}\}$ and $\{\Omega^{m}\}$, which consist
of a finite number $\varpi^{m}$ and $\Omega^{m}$ possessing the
following properties for each small number $\lambda>0$ and any
point $x^{m}\in \bar{\Omega}$:
\begin{enumerate}
    \item Denoting by $\mathfrak{B}_{r}(x^{m})$ the ball about $x^m$ of radius $r$, we have that
    $$
\varpi^{m}=\mathfrak{B}_{\lambda/2}(x^{m})\cap\bar{\Omega},\qquad
\Omega^{m}=\mathfrak{B}_{\lambda}(x^{m})\cap\bar{\Omega},
$$
and
$$
\overline{\varpi^{m}}\subset\Omega^{m}\subset\bar{\Omega},\qquad
\bigcup_{m}\varpi^{m}=\bigcup_{m}\Omega^{m}=\bar{\Omega}.
$$
    \item There exists a number $\mathcal{N}_{0}$ independent of $\lambda$
    such that the intersection of any $\mathcal{N}_{0}+1$ distinct
    $\Omega^{m}$ (and consequently any  $\mathcal{N}_{0}+1$ distinct
    $\varpi^{m}$) is empty.
\end{enumerate}

The index $m$ belongs to one of three
sets: $\mathfrak{M}$, $\mathfrak{N}_0$ or $\mathfrak{N}_1$,  where
\begin{align*}
m \in \mathfrak{M}\quad &\text{if }
\overline{\Omega^{m}}\cap\partial\Omega=\emptyset,\\
m \in\mathfrak{N}_0\quad &\text{if }
\overline{\varpi^{m}}\cap\{x=0\}\neq\emptyset,\\
m \in\mathfrak{N}_1\quad &\text{if }
\overline{\varpi^{m}}\cap\Gamma_1\neq\emptyset
.
\end{align*}
Finally, denoting $\partial\Omega^{m}=\partial\Omega\cap
\mathfrak{B}_{\lambda}(x^{m})$ and
$\Gamma_{1}^{m}=\Gamma_1\cap
\mathfrak{B}_{\lambda}(x^{m})$, we conclude that the covering $\{\varpi^{m}\}$ and
$\{\Omega^{m}\}$ define a partition of unity for the domain
$\Omega$.
Let $\xi^{m}(x):$ $\Omega\to [0,1]$ be a smooth function
such that
\begin{align*}
\xi^{m}(x)= 1 & \quad\text{if } x\in \overline{\varpi^{m}},\\
\xi^{m}(x)=0 & \quad\text{if } x\in \bar{\Omega}\backslash\overline{\Omega^{m}},\\
\xi^{m}(x)\in (0,1) & \quad \text{if } x\in \Omega^{m}\backslash\varpi^{m}.
\end{align*}
Moreover,
$$|D_{x}^{j}\xi^{m}|\leq C \lambda^{-|j|},\qquad 1\leq |j|,\quad 1\leq
\sum_{m}(\xi^{m})^{2}\leq \mathcal{N}_{0}.$$
Then, taking advantage of $\xi^{m}$, we define the function
\begin{equation}\label{3.8*}
\zeta^{m}:=\frac{\xi^{m}}{\sum_{j}(\xi^{j})^{2}}.
\end{equation}
The properties of the functions $\xi^{m}$ tell us that
 the
functions $\zeta^{m}$ vanish for $x\in \bar{\Omega}\backslash
\overline{\Omega^{m}}$; and in addition, $|D_{x}^{j}\zeta^{m}|\leq
C \lambda^{-|j|}$.
Hence, the product $\zeta^{m}\xi^{m}$ defines the partition of
unity via the formula
\begin{equation*}
\sum_{m}\zeta^{m}\xi^{m}=1.
\end{equation*}
After that,  in the spaces
$\C_{\alpha,0}^{k+\alpha,\frac{k+\alpha}{2}\nu}(\bar{\Omega}_{T}),$
$k=0,1,2,$ we introduce the norms associated with the covering
$\{\Omega^{m}\}$:
\[
\langle\langle w\rangle\rangle_{\C_{\alpha}^{k+\alpha,\frac{k+\alpha}{2}}(\bar{\Omega}_{T})}:=
\underset{m}{\sup}\|w\|_{\C_{\alpha}^{k+\alpha,\frac{k+\alpha}{2}}(\overline{\Omega^{m}}_{T})}.
\]
Repeating the arguments of Chapter 4 in \cite{LSU}, we can assert the
following.
\begin{proposition}\label{p3.1}
Let $w\in
\C_{\alpha,0}^{2+\alpha,\frac{2+\alpha}{2}}(\bar{\Omega}_{T})$. Then for $k=0,1,2,$
\begin{align*}
\|w\|_{\C([0,T],\C_{\alpha}^{k}(\bar{\Omega}))}&\leq
CT^{\frac{2+\alpha-k}{2}}\bigg[\langle\partial w/\partial t\rangle^{(\alpha/2)}_{t,0,\Omega_{T}}\\&+\bigg\langle
x \frac{\partial w}{\partial x}\bigg\rangle_{t,0,\Omega_{T}}^{(\frac{\alpha}{2})}
+\sum_{j=1}^{2}\bigg\langle
x \frac{\partial^{j} w}{\partial x^{j}}\bigg\rangle_{t,0,\Omega_{T}}^{(\frac{2+\alpha-j}{2})}
\bigg],
\end{align*}
and, besides, for $k=0,1,$ there holds
\begin{equation*}
\|w\|_{\C_{\alpha}^{k+\alpha,\frac{k+\alpha}{2}}(\bar{\Omega}_{T})}\leq
CT^{\frac{2+\alpha-k}{2}}\bigg[\langle\partial w/\partial t\rangle^{(\alpha/2)}_{t,0,\Omega_{T}}
+\sum_{j=1}^{2}\bigg\langle x
\frac{\partial^{j} w}{\partial x^{j}}\bigg\rangle_{t,0,\Omega_{T}}^{(\frac{2+\alpha-j}{2})}\bigg],
\end{equation*}
\begin{equation*}
\|\partial w/\partial t\|_{\C(\bar{\Omega}_{T})}\leq
CT^{\frac{\alpha}{2}}\langle\partial w/\partial t\rangle^{(\alpha/2)}_{t,0,\Omega_{T}},
\end{equation*}
where the positive constant $C$  does not depend on $T$.
\end{proposition}

Next, for $\tau\in[0,T]$ and an arbitrarily given $0<\mathfrak{d}<1$, we set
\begin{equation}\label{3.8}
\tau=\lambda^{2}\mathfrak{d}.
\end{equation}
\begin{proposition}\label{p3.2}
Let \eqref{3.8} hold. Then for $k=0,1,2,$ and  any function
 $w\in \C_{\alpha,0}^{k+\alpha,\frac{k+\alpha}{2}}(\bar{\Omega}_{\tau})$, there is the following norm equivalence:
\begin{equation*}
\langle\langle w\rangle\rangle_{\C_{\alpha}^{k+\alpha,\frac{k+\alpha}{2}}(\bar{\Omega}_{\tau})}\leq
\|w\|_{\C_{\alpha}^{k+\alpha,\frac{k+\alpha}{2}}(\bar{\Omega}_{\tau})}\leq
C\langle\langle w\rangle\rangle_{\C_{\alpha}^{k+\alpha,\frac{k+\alpha}{2}}(\bar{\Omega}_{\tau})},
\end{equation*}
where the positive constant $C$ is independent of $\lambda$ and
$\tau$.
\end{proposition}
\begin{proposition}\label{p3.3}
Let a function $\Phi_{m}(x)$ defined in $\Omega^{m}$ have
the property
\begin{equation*}
\bigg|\frac{\partial^{j} \Phi_{m}}{\partial x^{j}}(x)\bigg|\leq C\lambda^{-j},\quad 0\leq j\leq 2,
\end{equation*}
and the numbers $\tau$ and $\lambda$ are related via
\eqref{3.8}. Then for any function $w\in
\C_{\alpha,0}^{k+\alpha,\frac{k+\alpha}{2}}(\bar{\Omega}_{\tau})$,
$k=0,1,2$,
\[
\|\Phi_{m}w\|_{\C_{\alpha}^{k+\alpha,\frac{k+\alpha}{2}}(\bar{\Omega}^{m}_{\tau})}\leq
C\|w\|_{\C_{\alpha}^{k+\alpha,\frac{k+\alpha}{2}}(\bar{\Omega}^{m}_{\tau})},
\]
where the positive constant $C$ does not depend on $\lambda$ and
$\tau$.

Moreover, for $m\in\mathfrak{N}_{1}$ there holds
\[
C\|\Phi_{m}w\|_{\C_{\alpha}^{k+\alpha,\frac{k+\alpha}{2}}(\bar{\Omega}^{m}_{\tau})}\leq
\|w\|_{\C^{k+\alpha,\frac{k+\alpha}{2}}(\bar{\Omega}^{m}_{\tau})}\leq C
\|\Phi_{m}w\|_{\C_{\alpha}^{k+\alpha,\frac{k+\alpha}{2}}(\bar{\Omega}^{m}_{\tau})},
\]
and
\[
C\|w\|_{\C_{\alpha}^{k+\alpha,\frac{k+\alpha}{2}}(\bar{\Omega}^{m}_{\tau})}\leq
\|w\|_{\C^{k+\alpha,\frac{k+\alpha}{2}}(\bar{\Omega}^{m}_{\tau})}\leq C
\|w\|_{\C_{\alpha}^{k+\alpha,\frac{k+\alpha}{2}}(\bar{\Omega}^{m}_{\tau})}.
\]
\end{proposition}

\begin{proposition}\label{p3.4}
Let \eqref{3.8} hold and
$$w(x,t)=\sum_{m\in\mathfrak{M}\, \cup\mathfrak{N}\,_{0}\cup\mathfrak{N}\, _{1}}w^{m}(x,t),$$
where $w^{m}\in
\C_{\alpha,0}^{k+\alpha,\frac{k+\alpha}{2}}(\bar{\Omega}^{m}_{\tau})$ vanishes outside $\Omega^{m}$, 
$k=0,1,2$.
Then
\[
\langle\langle w\rangle\rangle_{\C_{\alpha}^{k+\alpha,\frac{k+\alpha}{2}}(\bar{\Omega}_{\tau})}\leq
C\underset{m\in\mathfrak{M}\,\cup\mathfrak{N}\,_{0}\cup\mathfrak{N}\,_{1}}{\sup}\|w^{m}\|_{\C_{\alpha}^{k+\alpha,\frac{k+\alpha}{2}}(\bar{\Omega}^{m}_{\tau})}.
\]
\end{proposition}


\subsection{Proof of Theorem \ref{t3.1}}
\label{s3.3}
\noindent
It is worth noting that the second part Theorem \ref{t3.1} which concerns to estimate \eqref{3.4} is a simple consequence of \eqref{3.2} and \eqref{3.3}. Indeed, in order to verify \eqref{3.4}, it is enough to substitute the new unknown function $\mathcal{U}(x,t)=x^{-s_0}u(x,t)$ to \eqref{1.1}-\eqref{1.3} and then, taking into account assumptions \textbf{H4}, apply the first part of Theorem \ref{t3.1} and estimate \eqref{3.2}. Thus, we are left to prove that \eqref{1.1}-\eqref{1.3} has a unique solution $u\in\C_{\alpha}^{2+\alpha,\frac{2+\alpha}{2}}(\bar{\Omega}_{T})$ satisfying \eqref{3.2}. To this end, following arguments of \S 7-8 in Chapter 4 \cite{LSU}, we will construct an operator $\mathfrak{R}$ so-called a \textit{regularizer} to \eqref{1.1}-\eqref{1.3} which allow us to construct the classical solution to this problem.

First of all, using Lemma \ref{l3.2} with $\mathcal{V}_0=u_0$ and $\mathcal{V}_1=\mathcal{L}u_{0}(x)|_{t=0}+f(x,0),$ $x\in\Omega$, we reduce \eqref{1.1}-\eqref{1.3} to the problem with homogenous initial data. Thus, we look for a solution to \eqref{1.1}-\eqref{1.3} in the form
\begin{equation}\label{3.9*}
u(x,t)=\mathcal{V}(x,t)+v(x,t)
\end{equation}
where $\mathcal{V}(x,t)$ is built in Lemma \ref{l3.2}. Coming to the new unknown function $v=v(x,t)$, it solves the following initial-boundary value problem
\begin{equation}\label{3.9}
\begin{cases}
\frac{\partial v}{\partial t}-\mathcal{L}v=\bar{f}\quad\text{in}\quad \Omega_{T},\\
v(x,0)=0\quad \text{in}\quad \bar{\Omega},\\
\mathcal{M} v=\bar{\psi}\quad\text{on}\quad \Gamma_{1,T},
\end{cases}
\end{equation}
where we put
\[
\bar{f}=f(x,t)-\frac{\partial\mathcal{V}}{\partial t}+\mathcal{L}\mathcal{V},\quad
\bar{\psi}=\psi(x,t)-\mathcal{V}|_{\Gamma_{1,T}}.
\]
It is worth noting that assumptions \textbf{H1-H3} and Lemma \ref{l3.2} provide the following smoothness
\begin{equation}\label{3.10}
\bar{f}\in\C_{\alpha,0}^{\alpha,\alpha/2}(\bar{\Omega}_{T})\quad \text{and}\quad \bar{\psi}\in\C_{0}^{1+\alpha,\frac{1+\alpha}{2}}(\Gamma_{1,T}).
\end{equation}
After that, for the sake of convenience, we rewrite problem \eqref{3.9} in more compact form
\begin{equation}\label{3.11}
\mathcal{A}v=h\quad\text{with}\quad h=(\bar{f},\bar{\psi}),
\end{equation}
where $\mathcal{A}$ is the linear operator defined by the left-hand side of \eqref{3.9}. In other words, $\mathcal{A}v=\{\mathcal{A}_{0}v,\mathcal{A}_1v|_{\Gamma_{1,T}}\}$, where $\mathcal{A}_{0}$ is given by the left-hand side of the equation in \eqref{3.9} while $\mathcal{A}_{1}$ is defined by the left-hand side of \textbf{3BC} in \eqref{3.9}.

Setting now for $m\in\mathfrak{M}\cup\mathfrak{N}_{0}\cup\mathfrak{N}_1$,
\begin{align*}
&a_{i}^{m}:=a_{i}(x^{m},0),\qquad \,\, b_{j}^{m}:=b_{j}(x^{m},0),\\
&f_{m}:=\xi^{m}(x)\bar{f}(x,t),\quad \psi_{m}:=\xi^{m}(x)\bar{\psi}(x,t),
\end{align*}
(where
 $\xi^{m},$ $\mathfrak{M}$, $\mathfrak{N}_0$, $\mathfrak{N}_1$ are defined in Subsection \ref{s3.2.2}),
 we introduce the functions $V^{m}:=V^{m}(x,t)$, $m\in\mathfrak{M}\cup\mathfrak{N}_{0}\cup\mathfrak{N}_1,$ as solutions to the following problems for $\tau\in(0,T]$ satisfying \eqref{3.8}.

If $m\in\mathfrak{M}$, then
\begin{equation}\label{3.12}
\begin{cases}
 \frac{\partial V^{m}}{\partial t}-x^{m}a_{0}^{m}\frac{\partial^{2}V^{m}}{\partial x^{2}}-a_{1}^{m}\frac{\partial V^{m}}{\partial x}=f_{m}\quad \text{in}\
\R_{\tau},\\
V^{m}(x,0)=0,\quad \text{in }
\mathbb{R}.
\end{cases}
\end{equation}
Instead, if $m\in\mathfrak{N}_1$, then
$V^{m}$
solves the initial-boundary value problem:
\begin{equation}
\label{3.13}
\begin{cases}
 \frac{\partial V^{m}}{\partial t}-x^{m}a_{0}^{m}\frac{\partial^{2}V^{m}}{\partial x^{2}}-a_{1}^{m}\frac{\partial V^{m}}{\partial x}=f_{m}\quad \text{in}\
\R_{+,\tau},\\
V^{m}(x,0)=0,\quad \text{in }
\mathbb{R}_{+},\\
b_{1}^{m}\frac{\partial V^{m}}{\partial x}+b_{2}^{m}V^{m}=\psi^{m}\quad\text{in}\quad \partial\R_{+,\tau}.
\end{cases}
\end{equation}
Finally, in the case of $m\in\mathfrak{N}_0$, we define $V^{m}$ as a solution of the Cauchy problem:
\begin{equation}
\label{3.14}
\begin{cases}
 \frac{\partial V^{m}}{\partial t}-xa_{0}^{m}\frac{\partial^{2}V^{m}}{\partial x^{2}}-a_{1}^{m}\frac{\partial V^{m}}{\partial x}=f_{m}\quad \text{in}\
\R_{+,\tau},\\
V^{m}(x,0)=0\quad \text{in }
\mathbb{R}_{+}.
\end{cases}
\end{equation}
It is easy to see that problems \eqref{3.12} and \eqref{3.13} are stated for usual (non-degenerate) parabolic linear equation, while problem \eqref{3.14} contains the degenerate parabolic equation similar to \eqref{1.1} with the coefficients satisfying to the Fiker condition.

\begin{definition}\label{d3.1}
Let $\tau\in(0,T]$.
An operator $\mathfrak{R}$ is called a \textit{regularizer}, on the time-interval $[0,\tau]$, if
\[
\mathfrak{R}:
\C_{\alpha,0}^{\alpha,\frac{\alpha}{2}}(\bar{\Omega}_{\tau})\times
\C_{0}^{1+\alpha,\frac{1+\alpha}{2}}(\Gamma_{1,\tau})\to
\C_{\alpha,0}^{2+\alpha,\frac{2+\alpha}{2}}(\bar{\Omega}_{\tau}),
\]
and
\begin{equation*}
\mathfrak{R}h=\sum_{m\in\mathfrak{M}\cup\mathfrak{N}\,_{1}\cup\mathfrak{N}\,_{0}}\zeta^{m}(x)V^{m}(x,t),
\end{equation*}
where the functions $\zeta^{m}(x)$ and $V^{m}(x,t)$ are given
in \eqref{3.8*} and \eqref{3.12}-\eqref{3.14}, respectively.
\end{definition}

The operator $\mathfrak{R}$ enables us to build an inverse
operator to $\mathcal{A}$.
At this point, we state the key
lemma.
\begin{lemma}\label{l3.3}
Let $\tau\in(0,T]$ satisfy \eqref{3.8} and assume
 the hypotheses of Theorem \ref{t3.1}.
Then, setting
\[
\mathcal{H}=\C_{\alpha,0}^{\alpha,\frac{\alpha}{2}}(\bar{\Omega}_{\tau})\times
\C_{0}^{1+\alpha,\frac{1+\alpha}{2}}(\Gamma_{1,\tau}),
\]
for any
\[
h\in\mathcal{H}\qquad\text{and}\qquad v\in
\C_{\alpha,0}^{2+\alpha,\frac{2+\alpha}{2}}(\bar{\Omega}_{\tau})\]
 the following hold:
 \begin{itemize}
    \item[{\bf (i)}] $\mathfrak{R}$ is a bounded operator:
    \begin{equation*}
\|\mathfrak{R}h\|_{\C_{\alpha}^{2+\alpha,\frac{2+\alpha}{2}}(\bar{\Omega}_{\tau})}\leq
C\|h\|_{\mathcal{H}},
    \end{equation*}
    where the positive constant $C$ is independent of $\lambda$
    and $\tau$.
    \smallskip
    \item[{\bf (ii)}] There exist operators $\mathfrak{T}_{1}$ and $\mathfrak{T}_{2}$ such that
    \begin{equation}\label{3.15}
\mathcal{A}\mathfrak{R}h=h+\mathfrak{T}_{1}h,\qquad
\mathfrak{R}\mathcal{A}v=v+\mathfrak{T}_{2}v,
\end{equation}
where
\[
\|\mathfrak{T}_{1}h\|_{\mathcal{H}}\leq\frac{1}{2}\|h\|_{\mathcal{H}}\quad\text{and}\quad
 \|\mathfrak{T}_{2}v\|_{\C_{\alpha}^{2+\alpha,\frac{2+\alpha}{2}}(\bar{\Omega}_{\tau})}
\leq\frac{1}{2}\|v\|_{\C_{\alpha}^{2+\alpha,\frac{2+\alpha}{2}}(\bar{\Omega}_{\tau})}.
\]
 \end{itemize}
\end{lemma}
\begin{proof} First, we verify statement (i) of this claim.
Simple linear changes of variables allow us to conclude that
the results of Lemma \ref{l3.1} above   and Theorem 6.1 in \cite{LSU} hold in the case of problems
\eqref{3.12}-\eqref{3.14}. Thus, collecting this results with Propositions
\ref{p3.1}-\ref{p3.4}, we arrive at the estimate
\begin{align*}
&\|\mathfrak{R}h\|_{\C_{\alpha}^{2+\alpha,\frac{2+\alpha}{2}}(\bar{\Omega}_{\tau})}
\leq
C\underset{m\in \mathfrak{M}\,\bigcup
\mathfrak{N\,}\,_{1}\bigcup \mathfrak{N\,}\,_{0}}{\sup}\|V^{m}\|_{\C_{\alpha}^{2+\alpha,\frac{2+\alpha}{2}}(\bar{\Omega}^{m}_{\tau})}
\notag\\
&\leq
C\Big[\underset{m\in \mathfrak{M}\,\bigcup
\mathfrak{N\, }\,_{1}\bigcup \mathfrak{N\,}\,_{0}}{\sup}\|\xi^{m}\bar{f}\|_{\C_{\alpha}^{\alpha,\frac{\alpha}{2}}(\bar{\Omega}^{m}_{\tau})}
+\underset{m\in\mathfrak{N\,}\,_{1}}{\sup}\|\xi^{m}\bar{\psi}\|_{\C^{1+\alpha,\frac{1+\alpha}{2}\nu}(\Gamma^{m}_{1,\tau})}
\Big]\notag\\
&\leq
C\|h\|_{\mathcal{H}},
\end{align*}
with the constant $C$ satisfying the assumptions of  the present lemma. Thus,
we complete the proof of (i).

Let us verify (ii) of this claim. Here we restrict ourself verification of the first equality in \eqref{3.15}. The second one in \eqref{3.15} is checked in the same manner with using analogous arguments from \cite[Chapter 4, $\S$ 7]{LSU}. Collecting definition of the operator $\mathcal{A}$ (see \eqref{3.11}) with the properties of $\zeta^{m}$ and $\xi^{m}$, we conclude that
\[
\mathcal{A}\mathfrak{R}h=\{\mathcal{A}_{0}\mathfrak{R}h,\mathcal{A}_{1}\mathfrak{R}h\big|_{\Gamma_{1,\tau}}\},
\]
where
\[
\mathcal{A}_{0}\mathfrak{R}h=\bar{f}+\mathfrak{T}_{1}^{1}h,\quad
\mathcal{A}_{1}\mathfrak{R}h=\bar{\psi}+\mathfrak{T}_{1}^{2}h,
\]
and
\begin{align*}
\mathfrak{T}_{1}^{1}h
&=a_{2}(x,t)\mathfrak{R}h-\sum_{m}\bigg[2xa_{0}(x,t)\frac{\partial\zeta^{m}}{\partial x}\frac{\partial V^{m}}{\partial x}
+xa_0(x,t)V^{m}\frac{\partial^{2}\zeta^{m}}{\partial x^{2}}\bigg]\\
&
-
\sum_{m\in\mathfrak{M}\,\bigcup\mathfrak{N\,}\,_{1}}\zeta^{m}[xa_0(x,t)-x^{m}a_{0}(x^{m},0)]\frac{\partial^{2}V^{m}}{\partial x^{2}}\\
&
-
\sum_{m\in\mathfrak{N\,}_{0}}\zeta^{m}x[a_0(x,t)-a_{0}(x^{m},0)]\frac{\partial^{2}V^{m}}{\partial x^{2}};
\\
\mathfrak{T}_{1}^{2}h
&=b_{2}(x,t)\mathfrak{R}h|_{\Gamma_{1,\tau}}+\sum_{m\in\mathfrak{N\,}\,_{1}}b_{1}(x,t)\frac{\partial\zeta^{m}}{\partial x}V^{m}|_{\Gamma_{1,\tau}}\\
&
+
\sum_{m\in\mathfrak{N\,}\,_{1}}\zeta^{m}[b_{1}(x,t)-b_{1}(x^{m},0)]\frac{\partial V^{m}}{\partial x}|_{\Gamma_{1,\tau}}.
\end{align*}
After that, appealing to Propositions \ref{p3.1}-\ref{p3.4}, Lemma \ref{l3.1}  above and Theorem 6.1. from \cite{LSU}, and taking into account assumptions \textbf{H2} with the following easily verified inequalities:
\[
\bigg\|x\zeta^{m}\frac{\partial^{2} V^{m}}{\partial x^{2}
}\bigg\|_{\C_{\alpha}^{\alpha,\frac{\alpha}{2}}(\bar{\Omega}^{m}_{\tau})}\leq
C[1+\mathfrak{d}^{\alpha/2}]\bigg\|x \frac{\partial^{2} V^{m}}{\partial x^{2}
}\bigg\|_{\C_{\alpha}^{\alpha,\frac{\alpha}{2}}(\bar{\Omega}^{m}_{\tau})}\, m\in\mathfrak{N}_{0},
\]
 \smallskip
\[
\bigg\|\zeta^{m}\frac{\partial^{2} V^{m}}{\partial x^{2}
}\bigg\|_{\C_{\alpha}^{\alpha,\frac{\alpha}{2}}(\bar{\Omega}^{m}_{\tau})}\leq
C[1+\mathfrak{d}^{\alpha/2}]\bigg\|\frac{\partial^{2} V^{m}}{\partial x^{2}
}\bigg\|_{\C_{\alpha}^{\alpha,\frac{\alpha}{2}}(\bar{\Omega}^{m}_{\tau})}\, m\in\mathfrak{M}\cup\mathfrak{N}_{1},
\]
 \smallskip
 \begin{align*}
\bigg\|x\frac{\partial \zeta^{m}}{\partial x}\frac{\partial
V^{m}}{\partial
x}\bigg\|_{\C_{\alpha}^{\alpha,\frac{\alpha}{2}}(\bar{\Omega}^{m}_{\tau})}&\leq
C\bigg[(\mathfrak{d}^{\frac{1+\alpha}{2}}+\mathfrak{d}^{\frac{1}{2}})\bigg\langle
x \frac{\partial
V^{m}}{\partial
x}\bigg\rangle_{t,0,\Omega^{m}_{\tau}}^{(\frac{1+\alpha}{2})}\\
&+\mathfrak{d}^{\frac{\alpha}{2}}\bigg\langle
x \frac{\partial^{2} V^{m}}{\partial x^{2}}\bigg\rangle_{t,0,\Omega^{m}_{\tau}}^{(\frac{\alpha}{2})}\bigg],
 \end{align*}
  \smallskip
 \begin{align*}
\bigg\|xV^{m}\frac{\partial^{2} \zeta^{m}}{\partial x^{2}}\bigg\|_{\C_{\alpha}^{\alpha,\frac{\alpha}{2}}(\bar{\Omega}^{m}_{\tau})}&\leq
C\bigg[(1+\lambda^{\alpha})\mathfrak{d}^{\frac{2+\alpha}{2}}
\bigg\langle
\frac{\partial V^{m}}{\partial t}\bigg\rangle_{t,0,\Omega^{m}_{\tau}}^{(\frac{\alpha}{2})}
\\&+
\mathfrak{d}^{\frac{1+\alpha}{2}}\bigg\langle x
\frac{\partial V^{m}}{\partial x}\bigg\rangle_{t,0,\Omega^{m}_{\tau}}^{(\frac{1+\alpha}{2})}
\bigg],
 \end{align*}
where the positive constant $C$
 is independent of $\lambda$ and $\tau$, we end up with the estimates:
\[
\|\mathfrak{T}_{1}^{1}h\|_{\C_{\alpha}^{\alpha,\alpha/2}(\bar{\Omega}_{\tau})}\leq \frac{1}{4}\|h\|_{\mathcal{H}},\quad
\|\mathfrak{T}_{1}^{2}h\|_{\C^{1+\alpha,\frac{1+\alpha}{2}}(\Gamma_{1,\tau})}\leq \frac{1}{4}\|h\|_{\mathcal{H}}
.
\]
In virtue of  $\mathfrak{T}_{1}h=\{\mathfrak{T}_{1}^{1}h,\mathfrak{T}_{1}^{2}h\}$, the last inequalities provide the first equality \eqref{3.15}. This finishes the proof of point (ii) and, consequence, the proof of Lemma \ref{l3.3}.
\end{proof}
Lemma \ref{l3.3} and, namely, relation \eqref{3.15} means that the operators $(I+\mathfrak{T}_1)$ and  $(I+\mathfrak{T}_2)$ (here $I$ is the identity) are invertible for a suitable small time $\tau$, and $(I+\mathfrak{T}_{j})^{-1},$ $j=1,2,$ are bounded. Hence, this arrives at the equalities
\[
\mathcal{A}\mathfrak{R}(I+\mathfrak{T}_{1})^{-1}h=h,\quad (I+\mathfrak{T}_{1})^{-1}\mathfrak{R}v=v,
\]
which imply that $\mathcal{A}$ has bounded right and left inverse operators such that
\[
\mathfrak{R}(I+\mathfrak{T}_{1})^{-1}=(I+\mathfrak{T}_{2})^{-1}\mathfrak{R}=\mathcal{A}^{-1}.
\]
Accordingly, the unique solution of \eqref{3.9} is given by
\[
v=\mathcal{A}^{-1}(\bar{f},\bar{\psi}),
\]
or returning to the function $u$ (see \eqref{3.9*}), we obtain a unique solution to the original problem \eqref{1.1}-\eqref{1.3} for $t\in[0,\tau]$:
\[
u(x,t)=\mathcal{V}(x,t)+\mathcal{A}^{-1}(\bar{f},\bar{\psi}).
\]
The estimate of the norm of $\mathcal{A}^{-1}$ follows from the corresponding  estimates in Lemma \ref{l3.3}. Finally, inequality \eqref{3.2} is a simple consequence of Lemmas \ref{l3.2} and \ref{l3.3}, and relations \eqref{3.9*}, \eqref{3.10}.

As a result, we have proved Theorem \ref{t3.1} for a small time interval $[0,\tau]$. In order to obtain this result in the general case, i.e., for $t\in[0,T]$, all we need is to extend the constructed solution on the interval $[i\tau,(i+1)\tau],$ $i=1,2,...$ while the entire $[0,T]$ is exhausted. To this end, we collect the technique from $\S 8$, Chapter 4 in \cite{LSU} with the arguments above in the case of $t\in[0,\tau]$. This allows us to have the classical solution $u$ on $[0,T]$, which satisfies inequality \eqref{3.2}. This completes the proof of Theorem \ref{t3.1}. \qed
\begin{remark}\label{r3.3}
Actually, with nonessential modifications in the arguments, the solvability of \eqref{1.1}-\eqref{1.3} in the weighted H\"{o}lder classes
 $E_{s}^{2+\alpha,\frac{2+\alpha}{2}}(\bar{\Omega}_{\tau})$, $s\in\R$, can be proved without requirement \eqref{3.3}. Indeed, as it follows from the proof of Theorem \ref{t3.1}, one should obtain the results similar to Lemma \ref{l3.1} in
 $E_{s}^{2+\alpha,\frac{2+\alpha}{2}}(\bar{\R}_{+,\tau})$. To this end, it is enough to modify arguments of Sections 2-4 in \cite{BV}.
\end{remark}


\subsection{Proof of Theorem \ref{t3.2}}
\label{s3.4}
\noindent Performing simple calculations, we deduce that
\[
\|u_{0}\|_{\C_{\alpha}^{2+\alpha}(\bar{\Omega})}\leq C\|u_{0}\|_{E_{s}^{2+\alpha}(\bar{\Omega})}\quad\text{and}\quad
\|f\|_{\C_{\alpha}^{\alpha,\alpha/2}(\bar{\Omega}_{T})}\leq C\|f\|_{E_{s}^{\alpha,\alpha/2}(\bar{\Omega}_{T})}.
\]
Thus, these relations and assumptions \textbf{H1,H3} provide the one-valued classical solvability of \eqref{1.1}-\eqref{1.3} in $\C_{\alpha}^{2+\alpha,\frac{2+\alpha}{2}}(\bar{\Omega}_{T})$. Taking into account this fact, we easily conclude that equality \eqref{3.5} follows immediately from the estimate
\begin{equation}\label{3.16}
\int\limits_{0}^{y}u^{2}(x,t)dx+
\int\limits_{0}^{t}u^{2}(0,s)ds\leq C_0 yt\rho(y)\, \text{for all}\, t\in[0,T],\, y\in(0,l^{\star}],
\end{equation}
with $l^{\star}=\min\{1/2,l\}$, and the bounded function $\rho(y)\in\C([0,l])$ satisfying the equality
\begin{equation}\label{3.17}
\rho(y)=O(y^{2\gamma-1})
\end{equation}
for some $\gamma\in(1/2,1)$ and for small $y$.
The positive constant $C_0$ in \eqref{3.16} depends only on $T, l$ and the corresponding norms of the coefficients $a_i$, $i=0,1,2,$ and the functions $u_0,$ $\psi$, $f$.

Therefore, here  we are left to verify \eqref{3.16}. To this end, we will exploit the following strategy consisting in two main steps. In this first one, we prove \eqref{3.16} for time $t\in[0,\tau]$ where
\begin{equation}\label{3.18}
\tau=\varepsilon^{\star}y^{\gamma}
\end{equation}
with some $\varepsilon^{\star}\in(0, \min\{1,T(l^{\star})^{-\gamma}\})$.

After that, we demonstrate how the obtained estimate can be extended to the whole time interval $[\tau,T]$.

\noindent\textit{Step 1:} Choosing some arbitrarily $y\in(0,l^{\star}]$, we multiply equation \eqref{1.1} by the solution $u$ and integrate over $(0,y)$. Standard calculations (as integrating by parts in the terms related with $\frac{\partial^{2}u}{\partial x^{2}}$ and $\frac{\partial u}{\partial x}$) yield
\begin{equation}\label{3.19}
\begin{aligned}
&\frac{1}{2}\frac{d}{dt}\int\limits_{0}^{y}u^{2}(x,t)dx+\int\limits_{0}^{y}xa_{0}(x,t)\big(\frac{\partial u}{\partial x}\big)^{2}dx
+\frac{\delta_{2}}{2}u^{2}(0,t)\\
&
=\int\limits_{0}^{y}\bigg[a_2(x,t)+\frac{1}{2}\frac{\partial a_{0}}{\partial x}(x,t)-\frac{1}{2}\frac{\partial a_{1}}{\partial x}(x,t)\bigg]u^{2}(x,t)dx
\\
&
+
y a_0(y,t)u(y,t)\frac{\partial u}{\partial x}(y,t)+\frac{1}{2}[a_1(y,t)-a_{2}(y,t)]u^{2}(y,t)
\\
&
-\int\limits_{0}^{y}x\frac{\partial a_0}{\partial x}(x,t)u(x,t)\frac{\partial u}{\partial x}(x,t)dx+\int\limits_{0}^{y}f(x,t)u(x,t)dx.
\end{aligned}
\end{equation}
Here, we used the facts: conditions \textbf{H6} and  the smoothness of the solution $u\in\C_{\alpha}^{2+\alpha,\frac{2+\alpha}{2}}(\bar{\Omega}_{T})$, which provide the equalities:
\begin{align*}
&\underset{x\to 0}{\lim} x a_{0}(x,t)u(x,t)\frac{\partial u}{\partial x}(x,t)=0,\\
&\underset{x\to 0}{\lim} [a_1(x,t)- a_{0}(x,t)]u^{2}(x,t)=\delta_2 u^{2}(0,t).
\end{align*}

Making use of \textbf{H1} for the second term in the left-hand side of \eqref{3.19}, and treating the two last terms in the right-hand side via Cauchy inequality, we arrive at
\begin{align*}
&\frac{1}{2}\frac{d}{dt}\int\limits_{0}^{y}u^{2}(x,t)dx+\delta_0\int\limits_{0}^{y}x\big(\frac{\partial u}{\partial x}\big)^{2}dx
+\frac{\delta_{2}}{2}u^{2}(0,t)\\
&
=\bigg[\underset{\bar{\Omega}_{T}}{\sup}\Bigg|a_2(x,t)+\frac{1}{2}\frac{\partial a_{0}}{\partial x}(x,t)-\frac{1}{2}\frac{\partial a_{1}}{\partial x}(x,t)\bigg|+\frac{1}{2}\\
&+\frac{y}{\varepsilon}\underset{\bar{\Omega}_{T}}{\sup}\Bigg|\frac{\partial a_{0}}{\partial x}(x,t)\bigg|\bigg]\int\limits_{0}^{y}u^{2}(x,t)dx
+
y\, \underset{\bar{\Omega}_{T}}{\sup}|a_0(y,t)||u(y,t)|\bigg|\frac{\partial u}{\partial x}(y,t)\bigg|
\\&+\frac{1}{2}[\delta_2+|a(y,t)|]u^{2}(y,t)
+\varepsilon\int\limits_{0}^{y}x\bigg[\frac{\partial u}{\partial x}(x,t)\bigg]^{2}dx
+\int\limits_{0}^{y}f^{2}(x,t)dx,
\end{align*}
or, choosing $\varepsilon=\frac{\delta_0}{2}$ and appealing to \textbf{H1, H5}, we have
\begin{equation}\label{3.20}
\begin{aligned}
&\frac{d}{dt}\int\limits_{0}^{y}u^{2}(x,t)dx+\int\limits_{0}^{y}x\big(\frac{\partial u}{\partial x}\big)^{2}dx
+u^{2}(0,t)\\
&
\leq C_{0}^{1}\Big\{ y|u(y,t)|\bigg|\frac{\partial u}{\partial x}(y,t)\bigg|+[1+y]u^{2}(y,t)\\&+
\|f\|^{2}_{E_{s-1}^{\alpha,\alpha/2}(\bar{\Omega}_{T})}y^{2s-1}+
\int\limits_{0}^{y}u^{2}(x,t)dx
\Big\},
\end{aligned}
\end{equation}
where the positive constant $C_{0}^{1}$ is defined with only the norms of the coefficients $a_i$ and is independent of $y$ and $\tau$.

In order to evaluate the first two terms in the right-hand side in this inequality, we appeal to estimate \eqref{3.2} and end up with the relations
\begin{align*}
&y|u(y,t)|\bigg|\frac{\partial u}{\partial x}(y,t)\bigg|\leq C_1 y\mathfrak{F}(u_0,f,\psi)\bigg[t\bigg|\frac{\partial u}{\partial t}(y,t)\bigg|+
u_0(y)\bigg]\\
&
\leq C_1 y[t+y^{s}]\mathfrak{F}_{s}^{2}(u_0,f,\psi),
\end{align*}
and
\[
u^{2}(y,t)\leq C_1 [t^{2}+y^{2s}]\mathfrak{F}^{2}_{s}(u_0,f,\psi),
\]
where we set
\[
\mathfrak{F}_{s}^{2}(u_0,f,\psi)=\bigg[\|u_0\|_{E_{s}^{2+\alpha}(\bar{\Omega})}+\|f\|_{E_{s-1}^{\alpha,\alpha/2}(\bar{\Omega}_{T})}
+\|\psi\|_{\C^{1+\alpha,\frac{1+\alpha}{2}}(\Gamma_{1,T})}
\bigg].
\]
Then, taking into account these estimates, and coming to \eqref{3.20}, we finally obtain
 \begin{equation*}
\begin{aligned}
&\frac{d}{dt}\int\limits_{0}^{y}u^{2}(x,t)dx+\int\limits_{0}^{y}x\big(\frac{\partial u}{\partial x}\big)^{2}dx
+u^{2}(0,t)\\
&
\leq C_{0}^{1}\max\{1,C_1\} \mathfrak{F}_{s}^{2}(u_0,f,\psi)[y^{2s-1}+yt+y^{s+1}+(t^{2}+y^{2s})(1+y)]\\&+
C_{0}^{1}\int\limits_{0}^{y}u^{2}(x,t)dx,
\end{aligned}
\end{equation*}
or, after integrating over $(0,t)$,
\begin{equation*}
\begin{aligned}
&\int\limits_{0}^{y}u^{2}(x,t)dx+\int\limits_{0}^{t}dz\int\limits_{0}^{y}x\big(\frac{\partial u}{\partial x}\big)^{2}dx
+\int\limits_{0}^{t}u^{2}(0,z)dz\\
&
\leq 
C_{0}^{1}\int\limits_{0}^{t}dz\int\limits_{0}^{y}u^{2}(x,z)dx+C_{0}^{1}\max\{1,C_1\} \mathfrak{F}_{s}^{2}(u_0,f,\psi)\\&
\times
(t[y^{2s-1}+yt+y^{s+1}+(t^{2}+y^{2s})(1+y)]+y^{2s+1}).
\end{aligned}
\end{equation*}
After that, the Gronwall inequality and relation \eqref{3.18} entail
\begin{equation}\label{3.21}
\int\limits_{0}^{y}u^{2}(x,t)dx+\int\limits_{0}^{t}dz\int\limits_{0}^{y}x\big(\frac{\partial u}{\partial x}\big)^{2}dx
+\int\limits_{0}^{t}u^{2}(0,z)dz\leq C_{0} t y \rho_{0}(y),
\end{equation}
where we set
\begin{equation}\label{3.21*}
\begin{aligned}
C_0&:= C_{0}^{1}\max\{1,C_1\} \mathfrak{F}_{s}^{2}(u_0,f,\psi)e^{TC_{0}^{1}},\\
\rho_{0}(y)&:=y^{2s-2}+y^{2s-1}+y^{2s}+y^{s}+2\varepsilon^{\star} y^{\gamma}+(\varepsilon^{\star})^{2}y^{2\gamma-1}.
\end{aligned}
\end{equation}
Since $1/2<\gamma<1$,  $s>1+\alpha$ and $y<1/2$, we easily conclude that $\rho_0(y)$ meets requirement \eqref{3.17}, and estimate \eqref{3.21}
leads to \eqref{3.16} if $t\in[0,\tau]$.

Moreover, inequality \eqref{3.21} and the mean value theorem provides key estimate which will be exploited in the second step of our arguments:
\begin{equation}\label{3.22}
u^{2}(0,t^{\star})\leq C_{0}y\rho_{0}(y)
\end{equation}
with some $t^{\star}\in(0,\tau)$. For simplicity consideration, we put $t^{\star}=\tau/2$.

\noindent \textit{Step 2:} At this point we proceed the technique which allows us to extend estimate \eqref{3.16} to the interval $[\tau,T]$.
First of all, we extend this estimate from $[0,\tau]$ to $[\tau,3\tau/2]$. To this end, we introduce a new variable $\bar{t}=t-\tau/2$, $\bar{t}\in[0,\tau],$ if $t\in[\tau/2,3\tau/2]$, and denote
\begin{equation*}
\begin{aligned}
 \bar{a}_{i}&:=a_{i}(x,\bar{t}+\tau/2),\,i=0,1,2,\quad\quad
\bar{b}_{j}:=b_{j}(x,\bar{t}+\tau/2),\,j=1,2,\\
\bar{\mathcal{L}}&:=x\bar{a}_{0}\frac{\partial^{2}}{\partial x^{2}}+\bar{a}_{1}\frac{\partial}{\partial x}+\bar{a}_{0},\qquad
\bar{\mathcal{M}}:=\bar{b}_{1}\frac{\partial}{\partial x}+\bar{b}_{2},\\
u&(x,\tau/2):=u_{\tau/2}(x),\qquad\qquad \bar{u}(x,\bar{t}):=u(x,\bar{t}+\tau/2),\\
 \bar{f}&(x,\bar{t}):=f(x,\bar{t}+\tau/2),\qquad\quad
\bar{\psi}(x,\bar{t}):=\psi(x,\bar{t}+\tau/2).
\end{aligned}
\end{equation*}
It is apparent that the coefficients $\bar{a}_{i},$ $\bar{b}_{j}$ and the functions $\bar{f},\bar{\psi}$ and $u_{\tau/2}(x)=\bar{u}(x,0)$ meet the requirements of Theorem \ref{t3.1} and condition \textbf{H6}, and, besides, $\bar{f}$ and $\bar{\psi}$ satisfy \textbf{H5}. Thus, recasting the arguments leading to estimate \eqref{3.21} (with  changing $u_{0}(x)$ by $u_{\tau/2}(x)$) in the case of the problem
\[
\begin{cases}
\frac{\partial\bar{u}}{\partial \bar{t}}-\bar{\mathcal{L}}\bar{u}=\bar{f}\quad\text{in}\quad \Omega_{\tau},
\\
\mathcal{M}\bar{u}=\bar{\psi}\quad \text{on}\quad\Gamma_{1,\tau},\\
\bar{u}(x,0)=u_{\tau/2}(x) \quad\text{in}\quad\bar{\Omega},
\end{cases}
\]
we conclude
\begin{equation}\label{3.23}
\int\limits_{0}^{y}\bar{u}^{2}(x,\bar{t})dx+\int\limits_{0}^{\bar{t}}dz\int\limits_{0}^{y}x\bigg(\frac{\partial \bar{u}}{\partial x}\bigg)^{2}dx+\int\limits_{0}^{\bar{t}}\bar{u}^{2}(0,z)dz\leq C_0 \bar{t}y\bar{\rho}_{0}(y).
\end{equation}
Here (in virtue of \eqref{3.22}) the function
\[
\bar{\rho}_{0}(y)=y^{2s-2}+(1+(\varepsilon^{\star})^{2})y^{2\gamma-1}+(1+2\varepsilon^{\star})y^{\gamma}+2y^{2\gamma}
\]
satisfies \eqref{3.17}. Besides, due to $y<1/2$ and $1/2<\gamma<1<s$, there holds
\[
\bar{\rho}_{0}(y)>\rho_{0}(y).
\]

Finally, collecting estimate \eqref{3.21} with \eqref{3.23} (where we come back from $\bar{u}$ to $u$), we deduce inequality \eqref{3.16} for $t\in[0,3\tau/2]$, with
$
\rho(y):=\bar{\rho}_{0}(y),
$ and $C_{0}$ defining in \eqref{3.21*}. Other words, we have extended estimate \eqref{3.16} from $[0,\tau]$ to $[\tau,3\tau/2]$. By the same token, we repeat the procedure to continue the obtain estimate on the other intervals until the entire $[0,T]$ is exhausted.
Finally, keeping in mind that $y$ was arbitrarily point in $[0,l^{\star}],$ estimate \eqref{3.16} holds for all $y\in[0,l^{\star}]$ and $t\in[0,T]$. This completes the proof of Theorem \ref{t3.2}.\qed


\subsection{Conservation Law}
\label{s4}
\noindent
In this section we discuss the properties of the solution to problem \eqref{1.1}-\eqref{1.3} which follow from Theorems \ref{t3.1} and \ref{t3.2}. Namely, we describe sufficient conditions on the given data in \eqref{1.1}-\eqref{1.3} which ensure the fulfillment of the conservation law to the solution $u\in\C_{\alpha}^{2+\alpha,\frac{2+\alpha}{2}}(\bar{\Omega}_{T})$:
\begin{equation}\label{4.1}
\frac{d}{dt}\int\limits_{\Omega}u(x,t)dx=0\quad \text{for all}\quad t\in[0,T].
\end{equation}
First, we state additional assumptions on the coefficients and the right-hand side of the model.

\noindent\textbf{H7:} Let $a_0\in\C_{\alpha}^{\alpha,\alpha/2}(\bar{\Omega}_{\tau})\cap\C([0,T],\C^{2}(\bar{\Omega})),$ and the following equalities hold:
\[
\begin{aligned}
&x\frac{\partial^{2}a_{0}}{\partial x^{2}}-\frac{\partial a_1}{\partial x}+2\frac{\partial a_{0}}{\partial x}+a_{2}=0\quad\text{for each }(x,t)\in\bar{\Omega}_{T},\\
&b_{1}(x,t)=x a_{0}(x,t),\,
b_{2}(x,t)=-a_{0}(x,t)-x\frac{\partial a_0}{\partial x}(x,t)+a_{1}(x,t) \,\text{on}\,\Gamma_{1,T},\\
&
\psi(l,t)=-\int\limits_{0}^{l}f(x,t)dx,\quad t\in[0,T].
\end{aligned}
\]
\begin{theorem}\label{t4.1}
Let $T>0$ be fixed and assumptions of Theorem \ref{t3.2} and \textbf{H7} hold. Then the classical solution $u\in\C_{\alpha}^{2+\alpha,\frac{2+\alpha}{2}}(\bar{\Omega}_{T})$ satisfies equality \eqref{4.1} for all $t\in[0,T]$.
\end{theorem}
\begin{proof}
First, we recall that Theorem \ref{t3.2} ensures the existence of  the classical solution $u\in\C_{\alpha}^{2+\alpha,\frac{2+\alpha}{2}}(\bar{\Omega}_{T})$ which vanishes at $x=0$ for all $t\in[0,T]$. Taking into account this fact, and integrating equation \eqref{1.1} over $\Omega$, we deduce
\[
\begin{aligned}
\frac{d}{dt}\int\limits_{0}^{l}u(x,t)dx&=\int\limits_{0}^{l}\Big[\frac{\partial^{2}}{\partial x^{2}}(x a_0)-\frac{\partial a_{1}}{\partial x}+a_2\Big]u(x,t)dx+\int\limits_{0}^{l}f(x,t)dx\\
&
+
\underset{x\to l}{\lim}\bigg[xa_{0}\frac{\partial u}{\partial x}-u\frac{\partial}{\partial x}(xa_{0})+a_1u\bigg].
\end{aligned}
\]
Then, taking advantage of \textbf{H7} to the terms in the right-hand side of this equality, we arrive at the desired relation \eqref{4.1}. This completes the proof of this claim.
\end{proof}
\begin{remark}\label{r4.1}
As it follows from the proof of Theorem \ref{t4.1}, the last three equalities in  \textbf{H7} can be changed by the more general condition:
\[
\int\limits_{0}^{l}f(x,t)dx+\underset{x\to l}{\lim}\bigg[x a_{0}\frac{\partial u}{\partial x}-u\frac{\partial}{\partial x}(x a_0)+a_1u\bigg]=0.
\]
\end{remark}


\subsection{A priori estimates in the Sobolev spaces}
\label{s5}
\noindent
Here we obtain a priori estimates in $L^2$ norms of solution to problem \eqref{1.1}-\eqref{1.3}.
We start with state additional assumption on the coefficients of equation \eqref{1.1}.

\noindent\textbf{H8:} We require that
\[
a_{1}(0,t)-a_{0}(0,t)\geq 0\quad\text{and}\quad \frac{\partial a_{0}}{\partial x}(0,t)<0\quad \text{for all}\quad t\in[0,T].
\]
\begin{lemma}\label{l5.1}
Let $T>0$ be fixed, assumptions \textbf{H1,H8} hold. Moreover, we assume that $a_{0}\in\C([0,T],\C^{2}(\bar{\Omega})),$ $a_{1}\in\C([0,T],\C^{1}(\bar{\Omega})),$  $b_{1}$ and $b_2$ meet requirements \textbf{H7}. If a solution $u$ of problem \eqref{1.1}-\eqref{1.2} with  homogenous condition \eqref{1.3} satisfies the equality
\begin{equation}\label{5.1}
\underset{x\to 0}{\lim} x a_{0}u\frac{\partial u}{\partial x}=0 \quad\text{for each}\quad t\in[0,T],
\end{equation}
then the following estimate is fulfilled
\begin{align*}
&\|u\|_{\C([0,T],L^{2}(\Omega))}+\|x^{1/2}\partial u/\partial x\|_{L^{2}(\Omega_{T})}+\|u(0,\cdot)\|_{L^{2}(0,T)}\\&
\leq C[\|u_{0}\|_{L^{2}(\Omega)}+\|f\|_{L^{2}(\Omega_{T})}]
\end{align*}
with the positive constant $C$ depending only on $T,$ $\Omega$ and corresponding norms of the coefficients $a_i$.
\end{lemma}
\begin{proof}
We multiply equation \eqref{1.1} by $u(x,t)$ and integrate over $\Omega=[0,l]$. Taking into account \textbf{H1}
and the smoothness of $a_i,$ we have
\[
\begin{aligned}
&\frac{1}{2}\frac{d}{dt}\int\limits_{0}^{l} u^{2}(x,t)dx+\delta_{0}\int\limits_{0}^{l}x\big(\frac{\partial u}{\partial x}\big)^{2}dx+
L_{0}(u)\\
&
\leq -L_{l}(u)+\int_{0}^{l}f^{2}dx
+\frac{1}{2}\int\limits_{0}^{l}\bigg[2-\frac{\partial a_1}{\partial x}+\frac{\partial^{2}(x a_0)}{\partial x^{2}}\bigg]u^{2}dx,
\end{aligned}
\]
where we put
\[
\begin{aligned}
L_0(u)&=\underset{x\to 0}{\lim}\bigg[xa_{0}u\frac{\partial u}{\partial x}+\frac{1}{2}\bigg(a_1-a_0-x\frac{\partial a_0}{\partial x}\bigg)u^{2}\bigg],\\
L_l(u)&=\underset{x\to l}{\lim}\bigg[-xa_{0}u\frac{\partial u}{\partial x}-\frac{1}{2}\bigg(a_1-a_0-x\frac{\partial a_0}{\partial x}\bigg)u^{2}\bigg].
\end{aligned}
\]
Assumptions on the coefficients $b_1$, $b_2$ and conditions \textbf{H8} with \eqref{5.1} provide inequalities:
\[
L_0(u)\geq 0\quad\text{and}\quad -L_{l}(u)<0.
\]
Appealing to these relations and conditions on $a_i$, we end up with the estimate
\[
\frac{d}{dt}\int\limits_{0}^{l} u^{2}(x,t)dx+\int\limits_{0}^{l}x\big(\frac{\partial u}{\partial x}\big)^{2}dx+
L_{0}(u)
\leq C\bigg[\int_{0}^{l}f^{2}dx
+\int\limits_{0}^{l}u^{2}dx\bigg],
\]
where the positive $C$ depends only on the norms of $a_i$, $i=0,1,2.$

Finally, Gronwall inequality arrives at the desired estimate.
\end{proof}
\begin{remark}\label{r5.1}
It is apparent that the classical solution  constructed in Theorem \ref{t3.1} satisfies assumption \eqref{5.1} of Lemma 5.1.
\end{remark}


\subsection{Conclusions of Theorems \ref{t3.1}-\ref{t4.1} and Lemma \ref{l5.1}}
\label{s6}
\noindent
Coming to original problem (\ref{c-1})--(\ref{c-3}), we conclude that this problem can be rewritten in the form of \eqref{1.1}-\eqref{1.3} with the following coefficients and the right-hand sides:
\[
\begin{aligned}
u(x,t)&=p(x,t),\quad f(x,t)=\psi(x,t)=0,\quad u_{0}(x)=p_{0}(x),\\
a_{0}(x,t)&=\frac{R_0+1-R_{0}x}{2N},\, a_{1}(x,t)=\frac{R_0+1-R_{0}x}{N}-x(R_0-1-R_0x),\\
a_2(x,t)&=-\frac{R_0}{N}-R_0+1+2R_{0}x,\quad b_1(x,t)=x\frac{R_0+1-R_{0}x}{2N},\\
b_2(x,t)&=-a_{0}(x,t)-x\frac{\partial a_0}{\partial x}(x,t)+a_{1}(x,t).
\end{aligned}
\]
These representations tell us that all assumptions of Theorems \ref{t3.1}, \ref{t3.2} and \ref{t4.1}, Lemma \ref{l5.1} hold, and we can assert the following.
\begin{theorem}\label{t6.1}
Let $T>0$ be fixed and  $p_0$ meet requirements of Theorem \ref{t3.1}. Then original problem (\ref{c-1})--(\ref{c-3})
 has a unique classical solution $p\in\C_{\alpha}^{2+\alpha,\frac{2+\alpha}{2}}(\bar{\Omega}_{T})$ satisfying estimates:
\[
\begin{aligned}
&\|p\|_{\C_{\alpha}^{2+\alpha,\frac{2+\alpha}{2}}(\bar{\Omega}_{T})}\leq C\|p_{0}\|_{\C_{\alpha}^{2+\alpha}(\bar{\Omega})},\\
&\|p\|_{\C([0,T],L^{2}(\Omega))}+\|\sqrt{x}\partial p/\partial x\|_{L^{2}(\Omega_{T})}+\|p(0,\cdot)\|_{L^{2}(0,T)}\leq
C\|p_0\|_{L^{2}(\Omega)}.
\end{aligned}
\]
If, in addition, $p_0$ meets assumptions of Theorem \ref{t3.2}, then the solution $p$ vanishes at $x=0$ for all $t\in[0,T]$ and the conservation law to this solution holds, i.e.
\[
\frac{d}{dt}\int\limits_{\Omega}p(x,t)dx=0 \quad\text{for all}\quad t\in[0,T].
\]
\end{theorem}

\end{document}